\documentclass[a4paper,fleqn,10pt]{article}
\usepackage{amsthm,amsmath,amssymb,amsfonts,color}
\newtheorem{theorem}{Theorem}

\numberwithin{equation}{section}
\numberwithin{lemma}{section}
\numberwithin{theorem}{section}
\numberwithin{corollary}{section}

\allowdisplaybreaks

\begin{document}
\title{On the discrete analogues of Appell function $F_1$}

\author{Ravi Dwivedi$^{1,}$\footnote{E-mail: dwivedir999@gmail.com}   \, and Vivek Sahai$^{2,}$\footnote{E-mail: sahai\_vivek@hotmail.com (Corresponding author)} \\ ${}^{1}$Department of Science, SAGEMMC, Jagdalpur, Bastar, CG, 494001, India; \\ ${}^{2}$Department of Mathematics and Astronomy, Lucknow University, \\ Lucknow 226007, India.}
\maketitle 
\begin{abstract}
In the present paper, two discrete forms of Appell function $F_1$ are introduced and studied. We examine the first discrete form in detail and give the results directly for the second. We study their regions of convergence, differential properties, integral representations, recursion relations, finite and infinite sums. The results and identities obtained in the paper are believed to be new and may find applications in various branches of science. 


\end{abstract}

\section{Introduction}
 Appell, in 1880, gave a set of four hypergeometric functions of two variables $F_1$, $F_2$, $F_3$ and $F_4$. These functions were named after him as Appell hypergeometric function or Appell series, \cite{kdf, emo}. After Appell, Humbert used the technique of getting degenerations (the functions obtained by allowing the limit in the parent function) of a hypergeometric function and obtained a list of seven degenerations of Appell functions, called Humbert functions $\phi_1$, $\phi_2$, $\phi_3$, $\psi_1$, $\psi_2$, $\varXi_1$ and $\varXi_2$, \cite{ph}. Later, Horn studied two variable functions systematically and produced a complete set of 34 functions including Appell functions and their degenerations, \cite{sk}. 

The physical appearance of these functions in applied mathematics, applied and theoretical physics, statistics, engineering and other fields increased  the interest of mathematicians and motivated them to explore these functions from different aspects. As a result, these functions have distinct generalizations and analogues, \emph{viz.} the $q$-Appell functions \cite{ds2,gr}, Appell functions with matrix arguments \cite{am, am1, nn}, Appell matrix functions \cite{ds1, ds5}, finite field version \cite{bts} and many more.  Recently, discrete analogues of special functions are attracting the mathematician and physicist most. The discrete analogue of Gauss hypergeometric function \cite{bc3}, the generalized hypergeometric function \cite{bc4, bc2}, Bessel functions \cite{bc1}, Legendre polynomial \cite{bc5}, to name a few, have appeared in the literature. 

In this paper, we introduce two discrete analogues of Appell function $F_1$. We elaborate the first one in detail and study its important properties \emph{viz.} special cases, limiting cases, region of convergence, difference-differential equations, integral representations, difference-differential formulae, finite and infinite sums and recursion relations. We also present the results for the second discrete analogue of $F_1$. The sectionwise treatment of the paper is as follows:

In Section~2, we list basic definitions and preliminaries that help in the understanding of the results obtained in the paper. In Section~3, we introduce discrete Appell functions and discuss its special cases, regions of convergence and difference-differential equations obeyed by the first discrete Appell function $\mathcal{F}_1^{(1)}$.  We also obtain integral representations for the first discrete Appell function $\mathcal{F}_1^{(1)}$. %{\color{red} 
The discrete Humbert functions obtained from $\mathcal{F}_1^{(1)}$ as degenerations are introduced. In Section~4, differential and difference formulae for discrete Appell function $\mathcal{F}_1^{(1)}$ are determined. In Section~5, we obtain finite and infinite summation results with the help of binomial and transformation formulae. In Section~6, we give the recursion formulae satisfied by $\mathcal{F}_1^{(1)}$. We also present a list of first and second order recursion relations of  $\mathcal{F}_1^{(1)}$. In Section~7, the other  discrete form of Appell function $F_1$ \emph{viz.} $\mathcal{F}_1^{(2)}$ is studied. We  present the results and theorems for $\mathcal{F}_1^{(2)}$ directly without proofs. Finally, in Section~8,  we have added a conclusion part which includes the scope of this work.

\section{Basic definitions and preliminaries} 
Throughout the paper, $\mathbb{N}$ and $\mathbb{C}$ will denote the set of natural numbers and complex numbers respectively. For a complex number $w$, the gamma function is defined in terms of Euler integral as \cite{edr}
\begin{align}
	\Gamma (w) = \int_{0}^{\infty} e^{-u} \, u^{w - 1} \, du, \quad \Re(w) > 0,
\end{align}
and for $v, w \in \mathbb{C}$, the beta function is defined by
\begin{align}
	\beta (v, w) = \int_{0}^{1} u^{v - 1} \, (1 - u)^{w - 1} \, du, \quad \Re(v), \Re (w) > 0.
\end{align}
These two basic special functions are related by the fundamental relation
\begin{align}
\beta (v, w) = \frac{\Gamma (v) \, \Gamma (w)}{\Gamma (v + w)}.
\end{align}	
The next milestone in the theory of special functions was the discovery of Gauss hypergeometric function. For $u, v, w \in \mathbb{C}$ and $w \ne 0, - 1, -2, \dots$, the Gauss function is defined by 
\begin{align}
	{}_2F_1 \left(\begin{array}{c}
	u, v\\
	w
	\end{array}; z\right) = \sum_{k = 0}^{\infty} \frac{(u)_k \, (v)_k}{(w)_k} \, \frac{z^k}{k !}, \quad \vert z \vert < 1, \label{24}
\end{align}	
where $(u)_k$ is the Pochhammer symbol given by
\begin{equation}
(u)_k = \begin{cases}
	1, & \text{ if } k = 0,\\
	u \, (u + 1) \cdots (u + k - 1), & \text{ if } k \ge 1.
\end{cases}
\end{equation}
There are several important identities involving Pochhammer symbol and we list some of them here, which are needed in rest of the paper. 
\begin{align}
	(u)_k & = \frac{\Gamma (u + k)}{\Gamma (u)};\\
	(u)_{k + l} & = (u)_k \, (u + k)_l =  (u)_l \, (u + l)_k;\\
	(u)_k & = (-1)^k \, (1 - u - k)_k;\\
	(- k)_l &= \frac{(-1)^l \, k !}{(k - l)!};\\
	(u)_{k - l} & = \frac{(-1)^l \, (u)_k}{(1 - u - k)_l}, \quad 0 \le l \le k.
\end{align}
The Gauss hypergeometric function \eqref{24} has  the following natural generalization called generalized hypergeometric function, defined by
\begin{align}
	{}_pF_q \left(\begin{array}{c}
		u_1, \dots, u_p\\
		v_1, \dots, v_q
	\end{array}; z\right) = \sum_{k = 0}^{\infty} \frac{(u_1)_k  \cdots (u_p)_k}{(v_1)_k  \cdots (v_q)_k} \, \frac{z^k}{k !}, \quad \vert z \vert < 1. 
\end{align}	
Also, it has two variable generalizations, known as Appell functions, given by
\begin{align}
	F_1 (u, v, v'; w; x, y) & = \sum_{k, l = 0}^{\infty} \frac{(u)_{k + l} \, (v)_k  \, (v')_l}{(w)_{k + l}} \, \frac{x^k \, y^l}{k ! \, l!}, \quad \vert x \vert, \, \vert y \vert < 1;\\
	F_2 (u, v, v'; w, w'; x, y) & = \sum_{k, l = 0}^{\infty} \frac{(u)_{k + l} \, (v)_k  \, (v')_l}{(w)_{k} \, (w')_{l}} \, \frac{x^k \, y^l}{k ! \, l!}, \quad \vert x \vert + \vert y \vert < 1;\\
	F_3 (u, u', v, v'; w; x, y) & = \sum_{k, l = 0}^{\infty} \frac{(u)_{k} \, (u')_l \, (v)_k  \, (v')_l}{(w)_{k + l}} \, \frac{x^k \, y^l}{k ! \, l!}, \quad \vert x \vert, \, \vert y \vert < 1;\\
	F_4 (u, v; w, w'; x, y) & = \sum_{k, l = 0}^{\infty} \frac{(u)_{k + l} \, (v)_{k + l}  }{(w)_{k} \, {(w')}_l} \, \frac{x^k \, y^l}{k ! \, l!}, \quad \sqrt{\vert x \vert} +  \sqrt{\vert y \vert} < 1.
\end{align}
Furthermore, the following generalized form of two variable function is known as the Kamp\'e de F\'eriet function:
\begin{align}
	& F_{{k_2}:{l_2}; {l'_2}} ^{{k_1}:{l_1}; {l'_1}}\left(\begin{array}{ccc}
		A: & B; & C\\
		D: &E; &F 
	\end{array}; x, \ y\right) =\displaystyle  \sum_{k,l = 0}^{\infty} \ \frac{\prod_{i=1}^{k_1} (a_i)_{k+l} \, \prod_{i=1}^{l_1} (b_i)_{k} \, \prod_{i=1}^{l'_1} (c_i)_{l}}{\prod_{i=1}^{k_2} (d_i)_{k+l}   \, \prod_{i=1}^{l_2} (e_i)_{k}  \, \prod_{i=1}^{l'_2} (f_i)_{l}} \ \frac{x^k \, y^l}{k! \, l!},\label{c1eq71}
\end{align}
where $A$ denotes the sequence of complex numbers $a_1$, \dots, $a_{k_1}$, etc. and $d_i, e_i, f_i \ne 0, -1, \dots$. The Kamp\'e de F\'eriet function converges \cite{sk}, for 
	\begin{align}
		(i) \ & k_1 + l_1 < k_2 + l_2 + 1, \, k_1 + l_1' < k_2 + l_2' + 1, \quad \vert x\vert < \infty, \,\vert y\vert < \infty, \text{ or } \nonumber \\
		(ii) \  &  k_1 + l_1 = k_2 + l_2 + 1, \, k_1 + l_1'= k_2 + l_2' + 1, \text{ and } \nonumber\\ 
		&  \begin{cases}
			\vert x\vert ^{\frac{1}{k_1 - k_2}} + \vert y\vert ^{\frac1{k_1 - k_2}}< 1, \text{ if }  k_1 > k_2,\\
			\vert x\vert < 1,  \vert y\vert < 1, \text{ if } k_1\leq k_2. 
		\end{cases}
	\end{align}  
Recently, the discrete analogue of special functions is being studied and is receiving much attention. % We now list some discrete version of  functions defined above. 
The generalized discrete hypergeometric function is defined by \cite{bc4}
 \begin{align}
 	{}_pF_q \left(\begin{array}{c}
 		u_1, \dots, u_p\\
 		v_1, \dots, v_q
 	\end{array}; t, k, z\right) = \sum_{m = 0}^{\infty} \frac{(u_1)_m  \cdots (u_p)_m}{(v_1)_m  \cdots (v_q)_m} \, \frac{(-1)^{m \, k} (-t)_{m \, k} \, z^m}{m !}, 
 \end{align}	
where
\begin{align}
	(-t)_{m \, k} & = 	k^{m \, k} \, \prod_{i = 0}^{k - 1} \left(\frac{-t + i}{k}\right)_m. \label{e219}
\end{align}
The discrete analogue of Gauss function,  Bessel function, Legendre polynomial, Gauss matrix function and generalized hypergeometric matrix function have been studied  in \cite{bc1}--\cite{bc5}, to name a few.

\section{Discrete Appell function $\mathcal{F}_1^{(1)}$}	
In this section, we introduce two distinct discrete forms of Appell function $F_1$. We elaborate one of the forms in detail and obtain its regions of convergence, special cases, difference equations, integral representations and its limiting cases. 

Let  $a$, $b_1$, $b_2$, $c$, $t$, $t_1$ and $t_2$ be complex numbers such that $\Re (c) \ne 0, -1, \dots$ and let $k, k_1, k_2 \in \mathbb{N}$. We define two discrete versions of Appell hypergeometric  function ${F}_1$, denoted by $\mathcal{F}^{(1)}_1$ and $\mathcal{F}^{(2)}_1$, as follows: 
	  \begin{align}
	 \mathcal{F}^{(1)}_1 & := \mathcal{F}^{(1)}_1(a, b_1, b_2; c; t_1, t_2, k_1, k_2, x, y)\nonumber\\ 
	 & = \sum_{m,n\geq0} \frac{(a)_{m+n} \, (b_1)_m \, (b_2)_n \, (-1)^{m k_1} \, (-1)^{n k_2} \, (-t_1)_{mk_1} \, (-t_2)_{nk_2}}{ (c)_{m+n} \, m! \, n!} \ x^m \, y^n; \label{3.1}
	  \end{align} 
	  \begin{align}
	  \mathcal{F}^{(2)}_1 & :=	\mathcal{F}^{(2)}_1(a, b_1, b_2; c; t, k, x, y)\nonumber\\
	  & = \sum_{m,n\geq0} \frac{(a)_{m+n} \, (b_1)_m \, (b_2)_n \, (-1)^{(m + n) k} \, (-t)_{(m + n)k}}{(c)_{m+n} \, m! \, n!} \ x^m \, y^n; \label{3.2}
	  \end{align} 
%	     \begin{align}
%	     	\mathcal{F}^{(3)}_1 & = \mathcal{F}^{(3)}_1(a, b_1, b_2; c; t, t, k, x, y)\nonumber\\
%	     	& = \sum_{m, n \geq 0} \frac{(a)_{m+n} \, (b_1)_m \, (b_2)_n \, (-1)^{(m + n) k} \, (-t)_{mk} \, (-t)_{nk}}{ (c)_{m+n} \, m! \, n!} \ x^m \, y^n. \label{3.3}
%	     \end{align} 
     Here, and in subsequent sections, we explore the discrete Appell  function $\mathcal{F}^{(1)}_1$ and later in the last section, we list results for the discrete  Appell function $\mathcal{F}^{(2)}_1$. 
     
	  For different values of $k_1$ and $k_2$, the discrete Appell function $\mathcal{F}^{(1)}_1$ reduces into classical functions. In particular, for $k_1 = 0 = k_2$, we have
	  \begin{align}
	  	\mathcal{F}^{(1)}_1(a, b_1, b_2; c; t_1, t_2, 0, 0, x, y) & 
%	  	= \sum_{m, n \geq 0} \frac{(a)_{m+n} \, (b_1)_m \, (b_2)_n}{ (c)_{m+n} \, m! \, n!} \ x^m \, y^n\nonumber\\ 	  	& 
	  	= F_1 (a, b_1, b_2; c; x, y).
	  \end{align}
     For $k_1 = 1, k_2 = 0$, it gives
      \begin{align}
     	 \mathcal{F}^{(1)}_1(a, b_1, b_2; c; t_1, t_2, 1, 0, x, y) 
%     	 & = \sum_{m, n \geq 0} \frac{(a)_{m+n} \, (b_1)_m \, (b_2)_n \, (-1)^{m} \, (-t_1)_m}{ (c)_{m+n} \, m! \, n!} \ x^m \, y^n\nonumber\\
     & = F_{{1}:{0}; {0}} ^{{1}:{2}; {1}}\left(\begin{array}{ccc}
      	a: & b_1, - t_1; & b_2\\
      	c: & - ; & - 
      \end{array}; -x, \ y\right).
      \end{align}
For $k_1 = 0, k_2 = 1$, the discrete function $ \mathcal{F}^{(1)}_1$ reduces into
  \begin{align}
  	\mathcal{F}^{(1)}_1(a, b_1, b_2; c; t_1, t_2, 0, 1, x, y) 
%  	& = \sum_{m, n \geq 0} \frac{(a)_{m+n} \, (b_1)_m \, (b_2)_n \, (-1)^{n} \, (-t_2)_n}{ (c)_{m+n} \, m! \, n!} \ x^m \, y^n\nonumber\\
  	& = F_{{1}:{0}; {0}} ^{{1}:1; {2}}\left(\begin{array}{ccc}
  		a: & b_1; & b_2, -t_2\\
  		c: & - ; & - 
  	\end{array}; x, \ -y\right).
  \end{align}
Similarly, for $k_1 = k_2 = 1$, we have 
  \begin{align}
  	&\mathcal{F}^{(1)}_1(a, b_1, b_2; c; t_1, t_2, 1, 1, x, y)
%  	\nonumber\\
%  	 & = \sum_{m, n \geq 0} \frac{(a)_{m+n} \, (b_1)_m \, (b_2)_n \, (-1)^{m} \, (-t_1)_m \, (-1)^{n} \, (-t_2)_n}{ (c)_{m+n} \, m! \, n!} \ x^m \, y^n\nonumber\\
  	 = F_{{1}:{0}; {0}} ^{{1}:{2}; {2}}\left(\begin{array}{ccc}
  		a: & b_1, - t_1; & b_2, -t_2\\
  		c: & - ; & - 
  	\end{array}; -x, \ -y\right).
  \end{align}
First, we examine the convergence of discrete Appell function $\mathcal{F}^{(1)}_1$. Let $\mathcal{A}_{m, n} x^m \, y^n$ be the general term of the discrete function $\mathcal{F}^{(1)}_1$. Then
  \begin{align}
  &	\left \vert \mathcal{A}_{m, n} x^m \, y^n \right \vert\nonumber\\
   & = \left \vert \frac{(a)_{m+n} \, (b_1)_m \, (b_2)_n \, (-t_1)_{mk_1} \, (-t_2)_{nk_2}}{ (c)_{m+n} \, m! \, n!} \ x^m \, y^n \right \vert\nonumber\\
  	& < \left \vert \frac{\Gamma (c)}{\Gamma (a) \, \Gamma (b_1) \, \Gamma (b_2) \, \Gamma (-t_1) \, \Gamma (-t_2)} \right \vert \nonumber\\
  	& \quad \times \left \vert \frac{\Gamma (a + m + n) \, \Gamma (b_1 + m) \, \Gamma (b_2 + n) \, \Gamma (-t_1 + mk_1) \, \Gamma (-t_2 + nk_2)}{\Gamma (c + m + n) \, \Gamma (m + 1) \, \Gamma (n + 1)} \right \vert \, \vert x \vert^m \, \vert y\vert^n.\label{2.9}
  \end{align}
Using the Stirling formula 
\begin{align}
	\lim_{n \to \infty} \Gamma (\lambda + n) = \sqrt (2 \pi) \, n^{\lambda + n - \frac{1}{2}} \, e^{-n},
\end{align}
for large values of $m$ and $n$, we write \eqref{2.9} as 
\begin{align}
	&	\left \vert \mathcal{A}_{m, n} x^m \, y^n \right \vert\nonumber\\
	& < \left \vert \frac{\Gamma (c)}{\Gamma (a) \, \Gamma (b_1) \, \Gamma (b_2) \, \Gamma (-t_1) \, \Gamma (-t_2)} \right \vert \nonumber\\
	& \quad \times \left \vert 2 \pi \, (m + n)^{a - c} \, m^{b_1 - 1} \, n^{b_2 - 1} \, (mk_1)^{mk_1 - t_1 - \frac{1}{2}} \, (nk_2)^{nk_2 - t_2 - \frac{1}{2}} e^{- (mk_1 + nk_2)} \right \vert \, \vert x \vert^m \, \vert y\vert^n.
\end{align}
Let $N > \left \vert \frac{2 \pi \, \Gamma (c)}{\Gamma (a) \, \Gamma (b_1) \, \Gamma (b_2) \, \Gamma (-t_1) \, \Gamma (-t_2)} \right \vert$. Then
\begin{align}
	&	\left \vert \mathcal{A}_{m, n} x^m \, y^n \right \vert < \frac{N \, (mk_1)^{mk_1 - t_1 - \frac{1}{2}} \, (nk_2)^{nk_2 - t_2 - \frac{1}{2}}}{ (m + n)^{c - a} \, m^{1 - b_1} \, n^{1 - b_2} \, e^{(mk_1 + nk_2)}} \, \vert x \vert^m \, \vert y\vert^n.
\end{align}
For $k_1, k_2 \in \mathbb{N}, t_1, t_2 \in \mathbb{C}$ and $\vert x\vert < 1, \vert y \vert < 1$, $\left \vert \mathcal{A}_{m, n} x^m \, y^n \right \vert \to 0$ as $m, n \to \infty$. Hence the discrete function $\mathcal{F}^{(1)}_1$ converges absolutely. 

Next, we discuss the difference equation obeyed by the discrete function $\mathcal{F}^{(1)}_1$. Let $\Delta_t$ be the forwarding operator such that $\Delta_t f(t) = f(t + 1) - f(t)$ and let $\rho_t$ be the shifting operator defined as $\rho_t \, f (t) = f(t - 1)$. Then, the discrete analogue of $\theta = t \, \frac{d}{dt}$ is denoted  by $\Theta_t := t \, \rho_t \, \Delta_t$. Using the difference operator $\Theta_t$, following result can be easily obtained
\begin{align}
\Theta_t \, ((-1)^{nk} \, (-t)_{nk}) & =
 %t \, \rho \, \Delta ((-1)^{nk} \, (-t)_{nk})\nonumber\\
%& =  t \, \rho \, \Delta (t \, (t - 1) \, \cdots (t - nk + 2) \, (t - nk + 1))\nonumber\\
%& =  t \, \rho ((t + 1) \, t \, \cdots (t - nk + 3) \, (t - nk + 2)\nonumber\\
%& \quad  - t \, (t - 1) \, \dots (t - nk + 2) \, (t - nk + 1))\nonumber\\
%& = t \, \rho (t \, (t - 1) \, \dots (t - nk + 2)  (t + 1 - t + nk - 1))\nonumber \\
%& = n \, k \, t \, \rho ((-1)^{nk - 1} \, (-t)_{nk - 1} )\nonumber\\
%& =
 %n \, k \, t \, (-1)^{nk - 1} \, (-t + 1)_{nk - 1}
  n\, k \, (-1)^{nk} \, (-t)_{nk}. 
\end{align} 
Now
\begin{align}
	& \Theta_{t_1} \left(\frac{1}{k_1} \Theta_{t_1} + \frac{1}{k_2} \Theta_{t_2} + c - 1\right) \, \mathcal{F}^{(1)}_1\nonumber \\
	& = \sum_{m,n \geq 0} \frac{(a)_{m+n} \, (b_1)_m \, (b_2)_n \, (-1)^{m k_1} \, (-1)^{n k_2} \, (-t_1)_{mk_1} \, (-t_2)_{nk_2}}{ (c)_{m+n} \, m! \, n!} \ x^m \, y^n \, mk_1 \, (c + m + n - 1)\nonumber\\
%	& = k  \sum_{m \ge 1, n \ge 0} \frac{(a)_{m+n} \, (b_1)_m \, (b_2)_n \, (-1)^{m k_1} \, (-1)^{n k_2} \, (-t_1)_{mk_1} \, (-t_2)_{nk_2}}{ (c)_{m + n - 1} \, (m - 1)! \, n!} \ x^m \, y^n  \nonumber\\
	& = k_1 \, \sum_{m, n \ge 0} \frac{(a)_{m+n + 1} \, (b_1)_{m + 1} \, (b_2)_n \, (-1)^{(m + 1) k_1} \, (-1)^{n  k_2}\, (-t_1)_{(m + 1)k_1} \, (-t_2)_{nk_2}}{ (c)_{m + n} \, m! \, n!} \ x^{m + 1} \, y^n\nonumber\\
%	& = k \, \sum_{m, n \ge 0} (a + m + n) \, (b_1 + m) \, (-1)^k \, (-t_1)_k \, x \nonumber\\
%	& \quad \times \frac{(a)_{m+n} \, (b_1)_{m} \, (b_2)_n \, (-1)^{(m + n) k} \, (-t_1 + k)_{m k} \, (-t_2)_{nk}}{ (c)_{m + n} \, m! \, n!} \ x^{m} \, y^n\nonumber\\
	& = k_1 \, \sum_{m, n \ge 0} (a + m + n) \, (b_1 + m) \, (-1)^{k_1} \, (-t_1)_{k_1} \, x \, \rho_{t_1}^{k_1} \nonumber\\
	& \quad \times \frac{(a)_{m+n} \, (b_1)_{m} \, (b_2)_n \, (-1)^{m  k_1} \, (-1)^{n  k_2}\, (-t_1)_{m k_1} \, (-t_2)_{nk_2}}{ (c)_{m + n} \, m! \, n!} \ x^{m} \, y^n\nonumber\\
	& = k_1 \, (-1)^{k_1} \, (-t_1)_{k_1} \, x \, \rho_{t_1}^{k_1} \,  \left(\frac{1}{k_1} \Theta_{t_1} + \frac{1}{k_2} \Theta_{t_2} + a\right) \, \left(\frac{1}{k_1} \Theta_{t_1}  + b_1\right) \, \mathcal{F}^{(1)}_1.
\end{align}
Thus, we arrive at 
\begin{align}
	&\left[\Theta_{t_1} \left(\frac{1}{k_1} \Theta_{t_1} + \frac{1}{k_2} \Theta_{t_2} + c - 1\right)\right]\mathcal{F}^{(1)}_1 \nonumber\\
	& = \left[k_1 \, (-1)^{k_1} \, (-t_1)_{k_1} \, x \, \rho_{t_1}^{k_1} \,  \left(\frac{1}{k_1} \Theta_{t_1} + \frac{1}{k_2} \Theta_{t_2} + a\right) \, \left(\frac{1}{k_1} \Theta_{t_1}  + b_1\right) \right] \mathcal{F}^{(1)}_1.\label{1.15}
\end{align}	    
This is the first difference equation satisfied by $\mathcal{F}^{(1)}_1$.  Similarly, one can get the second difference equation as
\begin{align}
	&\left[\Theta_{t_2} \, \left(\frac{1}{k_1} \Theta_{t_1} + \frac{1}{k_2} \Theta_{t_2} + c - 1\right)\right]\mathcal{F}^{(1)}_1\nonumber\\
	& =  \left[k_2 \, (-1)^{k_2} \, (-t_2)_{k_2} \, y \, \rho_{t_2}^{k_2} \,  \left(\frac{1}{k_1} \Theta_{t_1} + \frac{1}{k_2} \Theta_{t_2} + a\right) \, \left(\frac{1}{k_2} \Theta_{t_2}  + b_2 \right) \right] \mathcal{F}^{(1)}_1.\label{1.16}
\end{align}	    
From Equation \eqref{1.15}, we have
  \begin{align}
  	\frac{\left(\frac{1}{k_1} \Theta_{t_1} + \frac{1}{k_2} \Theta_{t_2} + a\right)}{\left(\frac{1}{k_1} \Theta_{t_1} + \frac{1}{k_2} \Theta_{t_2} + c - 1\right)} \, \mathcal{F}^{(1)}_1 = \frac{\Theta_{t_1}}{\left(\frac{1}{k_1} \Theta_{t_1}  + b_1\right) \, k_1 \, (-1)^{k_1} \, (-t_1)_{k_1} \, x \, \rho_{t_1}^{k_1}} \,  \mathcal{F}^{(1)}_1.\label{117}
  \end{align}  
Similarly, Equation \eqref{1.16} yields
\begin{align}
	\frac{\left(\frac{1}{k_1} \Theta_{t_1} + \frac{1}{k_2} \Theta_{t_2} + a\right)}{\left(\frac{1}{k_1} \Theta_{t_1} + \frac{1}{k_2} \Theta_{t_2} + c - 1\right)} \, \mathcal{F}^{(1)}_1 = \frac{\Theta_{t_2}}{\left(\frac{1}{k_2} \Theta_{t_2}  + b_2\right) \, k_2 \, (-1)^{k_2} \, (-t_2)_{k_2} \, y \, \rho_{t_2}^{k_2}} \,  \mathcal{F}^{(1)}_1. \label{118}
\end{align} 
Equations \eqref{117} and \eqref{118} together give
\begin{align}
\frac{\Theta_{t_1}}{\left(\frac{1}{k_1} \Theta_{t_1}  + b_1\right) \, k_1 \, (-1)^{k_1} \, (-t_1)_{k_1} \, x \, \rho_{t_1}^{k_1}} \,  \mathcal{F}^{(1)}_1 = \frac{\Theta_{t_2}}{\left(\frac{1}{k_2} \Theta_{t_2}  + b_2\right) \, k_2 \, (-1)^{k_2} \, (-t_2)_{k_2} \, y \, \rho_{t_2}^{k_2}} \,  \mathcal{F}^{(1)}_1
\end{align} 
or
\begin{align}
\left[k_2 \, (-1)^{k_2} \, (-t_2)_{k_2} \, y \, \rho_{t_2}^{k_2} \, \Theta_{t_1} \left(\frac{1}{k_2} \Theta_{t_2}  + b_2\right) -
%	\right] \mathcal{F}^{(1)}_1 \nonumber\\ 	& = \left[
k_1 \, (-1)^{k_1} \, (-t_1)_{k_1} \, x \, \rho_{t_1}^{k_1} \, \Theta_{t_2} \left(\frac{1}{k_1} \Theta_{t_1}  + b_1\right)\right] \,  \mathcal{F}^{(1)}_1=0. 
\end{align}
This is the third difference equation obeyed by the discrete function $\mathcal{F}^{(1)}_1$.  Note that these difference equations generalize the set of differential equations for the classical Appell function $F_1$ as in the paper \cite{mj}. 

\subsection{Integral representations}
In this section, we find the integral representations of the first discrete Appell  function $\mathcal{F}^{(1)}_1$.
\begin{theorem}\label{t7}
Let $a$, $b_1$, $b_2$, $c$, $t_1$ and $t_2$ be complex numbers. Then for $\vert x\vert < 1$, $\vert y\vert < 1$ and $k_1, k_2 \in \mathbb{N}$, the discrete Appell function defined in \eqref{3.1} can be represented in the following integral forms:
	\begin{align}
 & \mathcal{F}^{(1)}_1(a, b_1, b_2; c; t_1, t_2, k_1, k_2, x, y)\nonumber\\
 & =\Gamma \left(\begin{array}{c}
 c\\
 a, c - a
 \end{array}\right) \int_{0}^{1} u^{a - 1} (1 - u)^{c - a - 1} {}_1 \mathcal{F}_0(b_1; - ; t_1,  k_1, u x) \, {}_1 \mathcal{F}_0(b_2; - ; t_2,  k_2, u y) du.\label{1.17}
	\end{align}
	\begin{align}
	& \mathcal{F}^{(1)}_1(a, b_1, b_2; c; t_1, t_2, k_1, k_2, x, y)\nonumber\\
	& =\Gamma \left(\begin{array}{c}
		c\\
		b_1, b_2, c - b_1 - b_2
	\end{array}\right) \iint u^{b_1 - 1} v^{b_2 - 1} (1-u-v)^{c -  b_1 - b_2 - 1} \nonumber\\
& \quad \times F_{{0}:{0}; {0}} ^{{1}:{k_1}; {k_2}}\left(\begin{array}{ccc}
	a: & \frac{- t_1}{k_1}, \dots, \frac{- t_1 + k_1 - 1}{k_1} ; & \frac{- t_2}{k_2}, \dots, \frac{- t_2 + k_2 - 1}{k_2}\\
	-: & - ; & - 
\end{array}; (-k_1)^{k_1} \, u x,  (-k_2)^{k_2} \, v y\right) du\,dv,\nonumber\\
	&\qquad \qquad \qquad \qquad
	u\geq 0, \ v\geq 0, \ 1-u-v\geq 0. \label{1.18}
\end{align}
\end{theorem}	
\begin{proof}
Using 
\begin{align}
\frac{(a)_{m+n}}{(c)_{m+n}}
 & = \Gamma \left(\begin{array}{c}
 	c\\
 	a, c - a
 \end{array}\right) \int_{0}^{1} u^{a + m + n - 1} \, (1-u)^{c - a - 1} \, du, \label{58}
\end{align}
one can write
	\begin{align}
	& \mathcal{F}^{(1)}_1(a, b_1, b_2; c; t_1, t_2, k_1, k_2, x, y)\nonumber\\
	& =\Gamma \left(\begin{array}{c}
		c\\
		a, c - a
	\end{array}\right) \sum_{m, n \geq 0} \int_{0}^{1} u^{a + m + n - 1} \, (1-u)^{c - a - 1} \nonumber\\
& \quad \times \frac{(b_1)_m \, (b_2)_n \, (-1)^{m k_1} \, (-1)^{n k_2} \, (-t_1)_{mk_1} \, (-t_2)_{nk_2}}{ m! \, n!} \ x^m \, y^n du\nonumber\\
& = \Gamma \left(\begin{array}{c}
	c\\
	a, c - a
\end{array}\right) \int_{0}^{1} u^{a - 1} \, (1 - u)^{c - a - 1} \sum_{m\geq 0} \frac{(b_1)_m  \, (-1)^{m k_1} \, (-t_1)_{m k_1} \, (u \, x)^m}{ m! } \nonumber\\
& \quad \times  \sum_{n\geq 0} \frac{(b_2)_n \, (-1)^{ n k_2} \,  (-t_2)_{nk_2} (uy)^n}{  n!} \, du \nonumber\\
& =\Gamma \left(\begin{array}{c}
	c\\
	a, c - a
\end{array}\right) \int_{0}^{1} u^{a - 1} (1 - u)^{c - a - 1} {}_1 \mathcal{F}_0(b_1; - ; t_1,  k_1, u x) \, {}_1 \mathcal{F}_0(b_2; - ; t_2,  k_2, u y) du.
\end{align}
To prove \eqref{1.18}, we use the integral formula, \cite{ds1}
\begin{align}
	&\frac{(b_1)_m \, (b_2)_n}{(c)_{m + n}}\nonumber\\
	&  = \Gamma \left(\begin{array}{c}
		c\\
		b_1, b_2, c - b_1 - b_2
	\end{array}\right) \iint u^{b_1 + m - 1} v^{b_2 + n - 1} (1-u-v)^{c -  b_1 - b_2 - 1} \, du \, dv,\nonumber\\
&\qquad \qquad \qquad \qquad
u\geq 0, \ v\geq 0, \ 1-u-v\geq 0,
\end{align}
and as such get
\begin{align}
	& \mathcal{F}^{(1)}_1(a, b_1, b_2; c; t_1, t_2, k_1, k_2, x, y)\nonumber\\
	& =\Gamma \left(\begin{array}{c}
		c\\
		b_1, b_2, c - b_1 - b_2
	\end{array}\right) \sum_{m, n \geq 0} \iint u^{b_1 + m - 1} v^{b_2 + n - 1} (1-u-v)^{c -  b_1 - b_2 - 1} \nonumber\\
& \quad \times \frac{(a)_{m + n} \, (-1)^{m k_1} \, (-1)^{n k_2} (-t_1)_{mk_1} \, (-t_2)_{nk_2} }{m! \, n!} \, x^m \, y^n \, du \, dv. 
\end{align}
The implication of \eqref{e219}
%\begin{align}
%	(-t_i)_{mk} & = k^{mk} \, \left(\frac{-t_i}{k}\right)_m \, \left(\frac{-t_i + 1}{k}\right)_m \cdots \left(\frac{-t_i + k - 1}{k}\right)_m,  \quad i = 1, 2,
%	%\nonumber\\
%%	(-t_2)_{nk} & = k^{nk} \, \left(\frac{-t_2}{k}\right)_n \, \left(\frac{-t_2 + 1}{k}\right)_n \cdots \left(\frac{-t_2 + k - 1}{k}\right)_n,
%\end{align}
gives
\begin{align}
		& \mathcal{F}^{(1)}_1(a, b_1, b_2; c; t_1, t_2, k_1, k_2, x, y)\nonumber\\
	& =\Gamma \left(\begin{array}{c}
		c\\
		b_1, b_2, c - b_1 - b_2
	\end{array}\right)  \iint u^{b_1  - 1} v^{b_2  - 1} (1-u-v)^{c -  b_1 - b_2 - 1}  \nonumber\\
	& \quad \times \sum_{m, n \geq 0}  \frac{\left(\frac{-t_1}{k_1}\right)_m \, \left(\frac{-t_1 + 1}{k_1}\right)_m \cdots \left(\frac{-t_1 + k_1 - 1}{k_1}\right)_m \, \left(\frac{-t_2}{k_2}\right)_n \, \left(\frac{-t_2 + 1}{k_2}\right)_n \cdots \left(\frac{-t_2 + k_2 - 1}{k_2}\right)_n }{m! \, n! }\nonumber\\
	& \quad \times  (a)_{m + n} \, (-1)^{m k_1} \, (-1)^{n k_2} \, (k_1^{k_1} u x)^m \, (k_2^{k_2} v y)^n \, du \, dv\nonumber\\
& = \Gamma \left(\begin{array}{c}
	c\\
	b_1, b_2, c - b_1 - b_2
\end{array}\right)  \iint u^{b_1  - 1} v^{b_2  - 1} (1-u-v)^{c -  b_1 - b_2 - 1} \nonumber\\	
	& \quad \times F_{{0}:{0}; {0}} ^{{1}:{k_1}; {k_2}}\left(\begin{array}{ccc}
		a: & \frac{- t_1}{k_1}, \dots, \frac{- t_1 + k_1 - 1}{k_1} ; & \frac{- t_2}{k_2}, \dots, \frac{- t_2 + k_2 - 1}{k_2}\\
		-: & - ; & - 
	\end{array}; (-k_1)^{k_1} \, u x,  (-k_2)^{k_2} \, v y\right) du\,dv.
\end{align}
This completes the proof.
\end{proof}
Several other integral representations of discrete Appell function $\mathcal{F}^{(1)}_1$ can be established using the integral form of gamma function. We list some integrals here for the discrete Appell function $\mathcal{F}^{(1)}_1$: 
\begin{align}
& \mathcal{F}^{(1)}_1(a, b_1, b_2; c; t_1, t_2, k_1, k_2, x, y)\nonumber\\
& = \frac{1}{\Gamma (a)} \int_{0}^{\infty} e^{-u} \, u^{a - 1} \nonumber\\
& \quad \times F_{{1}:{0}; {0}} ^{{0}:{k_1 + 1}; {k_2 + 1}}\left(\begin{array}{ccc}
	- : & b_1, \frac{- t_1}{k_1}, \dots, \frac{- t_1 + k_1 - 1}{k_1} ; & b_2, \frac{- t_2}{k_2}, \dots, \frac{- t_2 + k_2 - 1}{k_2}\\
	c: & - ; & - 
\end{array}; (-k_1)^{k_1} \, u x,  (-k_2)^{k_2} \, u y\right) du;\label{3.9}\\
& = \frac{1}{\Gamma (b_1)} \int_{0}^{\infty} e^{-u} \, u^{b_1 - 1} \nonumber\\
& \quad \times F_{{1}:{0}; {0}} ^{{1}:{k_1}; {k_2 + 1}}\left(\begin{array}{ccc}
	a : & \frac{- t_1}{k_1}, \dots, \frac{- t_1 + k_1 - 1}{k_1} ; & b_2, \frac{- t_2}{k_2}, \dots, \frac{- t_2 + k_2 - 1}{k_2}\\
	c: & - ; & - 
\end{array}; (-k_1)^{k_1} \, u x,  (-k_2)^{k_2} \, y\right) du;\\
& = \frac{1}{\Gamma (b_2)} \int_{0}^{\infty} e^{-v} \, v^{b_2 - 1} \nonumber\\
& \quad \times F_{{1}:{0}; {0}} ^{{1}:{k_1 + 1}; {k_2}}\left(\begin{array}{ccc}
	a : & b_1, \frac{- t_1}{k_1}, \dots, \frac{- t_1 + k_1 - 1}{k_1} ; & \frac{- t_2}{k_2}, \dots, \frac{- t_2 + k_2 - 1}{k_2}\\
	c: & - ; & - 
\end{array}; (-k_1)^{k_1} \,  x,  (-k_2)^{k_2} \, v y\right) dv\\
& = \frac{1}{\Gamma (-t_1)} \int_{0}^{\infty} e^{-u} \, u^{-t_1 - 1} \nonumber\\
& \quad \times F_{{1}:{0}; {0}} ^{{1}:{1}; {k_2 + 1}}\left(\begin{array}{ccc}
	a : & b_1 ; & b_2, \frac{- t_2}{k_2}, \dots, \frac{- t_2 + k_2 - 1}{k_2}\\
	c: & - ; & - 
\end{array}; (-u)^{k_1} \, x,  (-k_2)^{k_2} \, y\right) du;\\
& = \frac{1}{\Gamma (-t_2)} \int_{0}^{\infty} e^{-v} \, v^{-t_2 - 1} \nonumber\\
& \quad \times F_{{1}:{0}; {0}} ^{{1}:{k_1 + 1}; {1}}\left(\begin{array}{ccc}
	a : & b_1, \frac{- t_1}{k_1}, \dots, \frac{- t_1 + k_1 - 1}{k_1} ; & b_2\\
	c: & - ; & - 
\end{array}; (-k_1)^{k_1} \,  x,  (- v)^{k_2} \, y\right) dv.
\end{align}
To prove \eqref{3.9}, we use the identity $(a)_{m + n} = \frac{\Gamma (a + m + n)}{\Gamma (a)}$ and the integral of $\Gamma (a + m + n)$ as
\begin{align}
\Gamma (a + m + n) = \int_{0}^{\infty} e^{-u} \, u^{a + m + n - 1} \, du.
\end{align}
Similarly, the other shifted factorials of numerator can be used to get the remaining integrals. 

We now introduce the discrete analogue of Humbert functions $\phi^{(1)}_1$, $\phi^{(1)}_2$ and $\phi^{(1)}_3$ as follows:
\begin{align}
	&\phi^{(1)}_1 \left(a, b_1; c; t_1, t_2, k_1, k_2, x,  y\right)\nonumber\\
	 &= \sum_{m,n \geq 0} \frac{(a)_{m+n} \, (b_1)_m \,  \, (-1)^{m k_1} \, (-1)^{n k_2} \, (-t_1)_{mk_1} \, (-t_2)_{nk_2}}{ (c)_{m+n} \, m! \, n!} \ x^m \, y^n;
	\\[5pt]
	&\phi^{(1)}_2 \left( b_1, b_2; c; t_1, t_2, k_1, k_2, x,  y\right)\nonumber\\
	 & = \sum_{m,n \geq 0} \frac{(b_1)_{m} \, (b_2)_n \,  \, (-1)^{m k_1} \, (-1)^{n k_2} \, (-t_1)_{mk_1} \, (-t_2)_{nk_2}}{ (c)_{m+n} \, m! \, n!} \ x^m \, y^n
	\\[5pt]
	&\phi^{(1)}_3 \left( b_1; c; t_1, t_2, k_1, k_2, x,  y\right) \nonumber\\
	& = 	\sum_{m,n \geq 0} \frac{(b_1)_{m} \,  \, (-1)^{m k_1} \, (-1)^{n k_2} \, (-t_1)_{mk_1} \, (-t_2)_{nk_2}}{ (c)_{m+n} \, m! \, n!} \ x^m \, y^n.
\end{align}
It can be shown that the limiting cases of discrete Appell function $\mathcal{F}^{(1)}_1$ give the discrete Humbert functions $\phi^{(1)}_1$, $\phi^{(1)}_2$ and $\phi^{(1)}_3$ as
\begin{align}
	& \lim_{\varepsilon \to 0}  \mathcal{F}^{(1)}_1 \left(a, b_1, \frac{1}{\varepsilon}; c; t_1, t_2, k_1, k_2, x, \varepsilon \, y\right) \nonumber\\
%	& = \lim_{\varepsilon \to 0} \sum_{m,n \geq 0} \frac{(a)_{m+n} \, (b_1)_m \, \left(\frac{1}{\varepsilon}\right)_n \, (-1)^{m k_1} \, (-1)^{n k_2} \, (-t_1)_{mk_1} \, (-t_2)_{nk_2}}{ (c)_{m+n} \, m! \, n!} \ x^m \, (\varepsilon \, y)^n\nonumber\\
	& = 	 \sum_{m,n \geq 0} \frac{(a)_{m+n} \, (b_1)_m \,  \, (-1)^{m k_1} \, (-1)^{n k_2} \, (-t_1)_{mk_1} \, (-t_2)_{nk_2}}{ (c)_{m+n} \, m! \, n!} \ x^m \, y^n \, \lim_{\varepsilon \to 0} \varepsilon^n \left(\frac{1}{\varepsilon}\right)_n\nonumber\\
%	& = 	 \sum_{m,n \geq 0} \frac{(a)_{m+n} \, (b_1)_m \,  \, (-1)^{m k_1} \, (-1)^{n k_2} \, (-t_1)_{mk_1} \, (-t_2)_{nk_2}}{ (c)_{m+n} \, m! \, n!} \ x^m \, y^n\nonumber\\
	& = \phi^{(1)}_1 \left(a, b_1; c; t_1, t_2, k_1, k_2, x,  y\right).
\end{align}
Similarly
\begin{align}
	\lim_{\varepsilon \to 0}  \mathcal{F}^{(1)}_1 \left(\frac{1}{\varepsilon}, b_1, b_2 ; c; t_1, t_2, k_1, k_2, \varepsilon \, x, \varepsilon \, y\right) = \phi^{(1)}_2 \left( b_1, b_2; c; t_1, t_2, k_1, k_2, x,  y\right).
\end{align}
\begin{align}
	\lim_{\varepsilon \to 0}  \mathcal{F}^{(1)}_1 \left(\frac{1}{\varepsilon}, b_1,\frac{1}{\varepsilon} ; c; t_1, t_2, k_1, k_2, \varepsilon \, x, \varepsilon^2 \, y\right) = \phi^{(1)}_3 \left( b_1; c; t_1, t_2, k_1, k_2, x,  y\right).
\end{align} 
\section{Differential and difference formulae}
%Before, we start finding the differential or difference formulae, first we introduce differential and difference operators. 
Let $\Delta_{t} = f(t + 1) - f(t)$ be the difference operator  and let $\theta = x \frac{\partial}{\partial x}, \phi = y \frac{\partial}{\partial y}$ be the differential operators. Then, we have the following difference and differential formulae satisfied by discrete Appell function $\mathcal{F}^{(1)}_1$:
\begin{align}
&	(\Delta_{t_1})^r \mathcal{F}^{(1)}_1(a, b_1, b_2; c; t_1, t_2, 1, k_2, x, y) \nonumber\\
& = \frac{(a)_r \, (b_1)_r  \, x^r }{(c)_r}   \mathcal{F}^{(1)}_1(a + r, b_1 + r, b_2; c + r; t_1 , t_2, 1, k_2, x, y);\label{4.1}\\
& (\Delta_{t_2})^r \mathcal{F}^{(1)}_1(a, b_1, b_2; c; t_1, t_2, k_1, 1, x, y) \nonumber\\
& = \frac{(a)_r \, (b_2)_r \, y^r}{(c)_r}   \mathcal{F}^{(1)}_1(a + r, b_1, b_2 + r; c + r; t_1, t_2, k_1, 1, x, y);\label{4.2}\\
%& (\Theta_{t_1})^r \mathcal{F}^{(1)}_1(a, b_1, b_2; c; t_1, t_2, k_1, k_2, x, y) \nonumber\\
%& = \frac{(-1)^{rk_1} \, (a)_r \, (b_1)_r \, (-t_1)_{rk_1} \, x^r \, k_1^r}{(c)_r} \nonumber\\
%& \quad \times  \mathcal{F}^{(1)}_1(a + r, b_1 + r, b_2; c + r; t_1 - rk_1, t_2, k_1, k_2, x, y);\\
%& (\Theta_{t_2})^r \mathcal{F}^{(1)}_1(a, b_1, b_2; c; t_1, t_2, k_1, k_2, x, y) \nonumber\\
%& = \frac{(-1)^{rk_2} \, (a)_r \, (b_2)_r \, (-t_2)_{rk_2} \, y^r \, k_2^r}{(c)_r} \nonumber\\
%& \quad \times  \mathcal{F}^{(1)}_1(a + r, b_1, b_2 + r; c + r; t_1, t_2 - rk_2, k_1, k_2, x, y);\\
& (\theta)^r \mathcal{F}^{(1)}_1(a, b_1, b_2; c; t_1, t_2, k_1, k_2, x, y) \nonumber\\
& = \frac{(-1)^{rk_1} \, (a)_r \, (b_1)_r \, (-t_1)_{rk_1} \, x^r}{(c)_r} \mathcal{F}^{(1)}_1(a + r, b_1 + r, b_2; c + r; t_1 - rk_1, t_2, k_1, k_2, x, y);\\
& (\phi)^r \mathcal{F}^{(1)}_1(a, b_1, b_2; c; t_1, t_2, k_1, k_2, x, y) \nonumber\\
& = \frac{(-1)^{rk_2} \, (a)_r \, (b_2)_r \, (-t_2)_{rk_2} \, y^r}{(c)_r}  \mathcal{F}^{(1)}_1(a + r, b_1, b_2 + r; c + r; t_1, t_2 - rk_2, k_1, k_2, x, y).
\end{align}
 The proofs of these differential formulae are straightforward. We give here the proof of \eqref{4.1} only and to accomplish this we use the induction method. Using the identity $\Delta_{t_1} [(-1)^{m} \, (-t_1)_{m}] = m \, (-1)^{m - 1} \, (-t_1)_{m - 1}$, we get
 \begin{align}
 &(\Delta_{t_1}) \mathcal{F}^{(1)}_1(a, b_1, b_2; c; t_1, t_2, 1, k_2, x, y)\nonumber\\
& = \sum_{m,n \geq 0} \frac{(a)_{m+n} \, (b_1)_m \, (b_2)_n \, (-1)^{m  - 1} \, m \, (-t_1)_{m - 1} \, (-1)^{nk_2} \, (-t_2)_{nk_2}}{ (c)_{m+n} \, m! \, n!} \ x^m \, y^n\nonumber\\
%& = k \, \sum_{m \ge 1, n \geq 0} \frac{(a)_{m+n} \, (b_1)_m \, (b_2)_n \, (-1)^{(m + n) k - 1} \, (-t_1)_{mk - 1} \, (-t_2)_{nk}}{ (c)_{m+n} \, (m - 1)! \, n!} \ x^m \, y^n\nonumber\\
%& = k \sum_{m, n \geq 0} \frac{(a)_{m+n + 1} \, (b_1)_{m + 1} \, (b_2)_n \, (-1)^{(m + n + 1) k - 1}  \, (-t_1)_{mk + k - 1} \, (-t_2)_{nk}}{ (c)_{m+n + 1} \, m! \, n!} \ x^{m + 1} \, y^n\nonumber\\
& =  \frac{a \, b_1 \, x}{c} \sum_{m,n \geq 0} \frac{(a + 1)_{m+n} \, (b_1 + 1)_m \, (b_2)_n \, (-1)^{m } \, (-t_1)_{m} \, (-1)^{n k_2} (-t_2)_{nk_2}}{ (c + 1)_{m+n} \, m! \, n!} \ x^m \, y^n\nonumber\\
& =   x  \frac{(a)_1 \,( b_1)_1}{(c)_1} \mathcal{F}^{(1)}_1(a + 1, b_1 + 1, b_2; c + 1; t_1, t_2, 1, k_2, x, y).
 \end{align}
Thus the formula is true for $r = 1$. Assuming the formula to be valid for $r = p$, 
%that is,
%\begin{align} 
%&	(\Delta_{t_1})^p \mathcal{F}^{(1)}_1(a, b_1, b_2; c; t_1, t_2, 1, k_2, x, y) \nonumber\\
%& = \frac{(a)_p \, (b_1)_p  \, x^p }{(c)_p}   \mathcal{F}^{(1)}_1(a + p, b_1 + p, b_2; c + p; t_1 , t_2, 1, k_2, x, y).
%\end{align}
%Using this,
we have
\begin{align} 
	&	(\Delta_{t_1})^{p + 1} \mathcal{F}^{(1)}_1(a, b_1, b_2; c; t_1, t_2, 1, k_2, x, y) \nonumber\\
	& = \Delta_{t_1} [(\Delta_{t_1})^{p} \mathcal{F}^{(1)}_1(a, b_1, b_2; c; t_1, t_2, 1, k_2, x, y)] \nonumber\\
	& = \frac{(a)_p \, (b_1)_p  \, x^p }{(c)_p}  \Delta_{t_1} [ \mathcal{F}^{(1)}_1(a + p, b_1 + p, b_2; c + p; t_1 , t_2, 1, k_2, x, y)]\nonumber\\
	& = \frac{(a)_p \, (b_1)_p  \, x^p }{(c)_p} \, \frac{(a + p) \, (b_1 + p)  \, x }{(c + p)} \nonumber\\
	& \quad \times  \mathcal{F}^{(1)}_1 (a + p + 1, b_1 + p + 1, b_2; c + p + 1; t_1, t_2, 1, k_2, x, y)\nonumber\\
	& = \frac{(a)_{p + 1} \, (b_1)_{p + 1}  \, x^{p + 1} }{(c)_{p + 1}} \mathcal{F}^{(1)}_1 (a + p + 1, b_1 + p + 1, b_2; c + p + 1; t_1, t_2, 1, k_2, x, y).
\end{align}
Therefore, the result is true for every $r \in \mathbb{N}$. It completes the proof. 
%The particular cases of these formulae lead to the difference and differential formulae for  Kamp\'e de F\'eriet type hypergeometric function as:
% \begin{align}
% 	&	(\Delta_{t_1})^r \mathcal{F}^{(1)}_1(a, b_1, b_2; c; t_1, t_2, 1, x, y) \nonumber\\
% 	& = \frac{(a)_r \, (b_1)_r \, x^r}{(c)_r}   \mathcal{F}^{(1)}_1(a + r, b_1 + r, b_2; c + r; t_1, t_2, 1, x, y);\nonumber\\
% 	& (\Delta_{t_2})^r \mathcal{F}^{(1)}_1(a, b_1, b_2; c; t_1, t_2, 1, x, y) \nonumber\\
% 	& = \frac{(a)_r \, (b_2)_r \, y^r}{(c)_r}   \mathcal{F}^{(1)}_1(a + r, b_1, b_2 + r; c + r; t_1, t_2, 1, x, y);\nonumber\\
% 	& (\Theta_{t_1})^r \mathcal{F}^{(1)}_1(a, b_1, b_2; c; t_1, t_2, 1, x, y) \nonumber\\
% 	& = \frac{(-1)^{r} \, (a)_r \, (b_1)_r \, (-t_1)_{r} \, x^r \, }{(c)_r} \, \mathcal{F}^{(1)}_1(a + r, b_1 + r, b_2; c + r; t_1 - r, t_2, 1, x, y);\\
% 	& (\Theta_{t_2})^r \mathcal{F}^{(1)}_1(a, b_1, b_2; c; t_1, t_2, 1, x, y) \nonumber\\
% 	& = \frac{(-1)^{r} \, (a)_r \, (b_2)_r \, (-t_2)_{r} \, y^r}{(c)_r} \, \mathcal{F}^{(1)}_1(a + r, b_1, b_2 + r; c + r; t_1, t_2 - r, 1, x, y);\\
% 	& (\theta)^r \mathcal{F}^{(1)}_1(a, b_1, b_2; c; t_1, t_2, 1, x, y) \nonumber\\
% 	& = \frac{(-1)^{r} \, (a)_r \, (b_1)_r \, (-t_1)_{r} \, x^r}{(c)_r} \, \mathcal{F}^{(1)}_1(a + r, b_1 + r, b_2; c + r; t_1 - r, t_2, 1, x, y);\\
% 	& (\phi)^r \mathcal{F}^{(1)}_1(a, b_1, b_2; c; t_1, t_2, 1, x, y) \nonumber\\
% 	& = \frac{(-1)^{r} \, (a)_r \, (b_2)_r \, (-t_2)_{r} \, y^r}{(c)_r} \, \mathcal{F}^{(1)}_1(a + r, b_1, b_2 + r; c + r; t_1, t_2 - r, 1, x, y).
% \end{align}

Besides, some other differential formulas are as follows
\begin{align}
	& \left(\frac{\partial}{\partial x}\right)^r \left[x^{b_1 + r - 1} \mathcal{F}^{(1)}_1(a, b_1, b_2; c; t_1, t_2, k_1, k_2, x, y)\right]\nonumber\\
	&\qquad = x^{b_1 - 1} \, (b_1)_r \, \mathcal{F}^{(1)}_1(a, b_1 + r, b_2; c; t_1, t_2, k_1, k_2, x, y);\label{4.14}\\
		& \left(\frac{\partial}{\partial y}\right)^r [y^{b_2 + r - 1} \mathcal{F}^{(1)}_1(a, b_1, b_2; c; t_1, t_2, k_1, k_2, x, y)]\nonumber\\
	&\qquad  = y^{b_2 - 1} \, (b_2)_r \, \mathcal{F}^{(1)}_1(a, b_1, b_2 + r; c; t_1, t_2, k_1, k_2, x, y);\\
		& \left(\frac{\partial}{\partial x}\right)^r [x^{a + r - 1} \mathcal{F}^{(1)}_1(a, b_1, b_2; c; t_1, t_2, k_1, k_2, x, xy)]\nonumber\\
	&\qquad  = x^{a - 1} \, (a)_r \, \mathcal{F}^{(1)}_1(a + r, b_1, b_2; c; t_1, t_2, k_1, k_2, x, xy);\\
		& \left(\frac{\partial}{\partial y}\right)^r [y^{a + r - 1} \mathcal{F}^{(1)}_1(a, b_1, b_2; c; t_1, t_2, k_1, k_2, xy, y)]\nonumber\\
	&\qquad  = y^{a - 1} \, (a)_r \, \mathcal{F}^{(1)}_1(a + r, b_1, b_2; c; t_1, t_2, k_1, k_2, xy, y);\\
		& \left(\frac{\partial}{\partial x}\right)^r [x^{c - 1} \mathcal{F}^{(1)}_1(a, b_1, b_2; c; t_1, t_2, k_1, k_2, x, xy)]\nonumber\\
	&\qquad  = (-1)^r \, (1 - c)_r \, x^{c - r - 1}  \mathcal{F}^{(1)}_1(a, b_1, b_2; c - r; t_1, t_2, k_1, k_2, x, xy);\\
		& \left(\frac{\partial}{\partial y}\right)^r [y^{c - 1} \mathcal{F}^{(1)}_1(a, b_1, b_2; c; t_1, t_2, k_1, k_2, xy, y)]\nonumber\\
	&\qquad  = (-1)^r \, y^{c - r - 1} \, (1 - c)_r \, \mathcal{F}^{(1)}_1(a, b_1, b_2; c - r; t_1, t_2, k_1, k_2, xy, y).\label{413}
\end{align}
The proofs of \eqref{4.14}-\eqref{413} can be done using some basic differentiation and manipulations. To prove \eqref{4.14}, we start with 
\begin{align}
	& \frac{\partial}{\partial x} \left[x^{b_1} \mathcal{F}^{(1)}_1(a, b_1, b_2; c; t_1, t_2, k_1, k_2, x, y)\right]\nonumber\\
%	& = \sum_{m, n \geq 0} \frac{(a)_{m+n} \, (b_1)_m \, (b_2)_n \, (-1)^{m k_1} \, (-1)^{n k_2} \, (-t_1)_{mk_1} \, (-t_2)_{nk_2}}{ (c)_{m+n} \, m! \, n!} \, \frac{\partial}{\partial x} x^{b_1 + m} \, y^n\nonumber\\
	& = \sum_{m, n \geq 0} \frac{(a)_{m+n} \, (b_1)_m \, (b_2)_n \, (-1)^{m k_1} \, (-1)^{n k_2} \, (-t_1)_{mk_1} \, (-t_2)_{nk_2}}{ (c)_{m+n} \, m! \, n!} \, (b_1 + m) \, x^{b_1 + m - 1} \, y^n\nonumber\\
%	& = x^{b_1 - 1} \, b_1 \sum_{m, n \geq 0} \frac{(a)_{m+n} \, (b_1 + 1)_m \, (b_2)_n \, (-1)^{m k_1} \, (-1)^{n k_2} \, (-t_1)_{mk_1} \, (-t_2)_{nk_2}}{ (c)_{m+n} \, m! \, n!} \,  x^{m} \, y^n\nonumber\\
	& = x^{b_1 - 1} \, b_1 \mathcal{F}^{(1)}_1(a, b_1 + 1, b_2; c; t_1, t_2, k_1, k_2, x, y). 
\end{align}
This shows that the result is true for $r = 1$. Considering the result to be true for $r = p$, 
%\begin{align}
%& \left(\frac{\partial}{\partial x}\right)^p \left[x^{b_1 + p - 1} \mathcal{F}^{(1)}_1(a, b_1, b_2; c; t_1, t_2, k_1, k_2, x, y)\right]\nonumber\\
%& = x^{b_1 - 1} \, (b_1)_p \ \mathcal{F}^{(1)}_1(a, b_1 + p, b_2; c; t_1, t_2, k_1, k_2, x, y)
%\end{align}
we have
\begin{align}
	& \left(\frac{\partial}{\partial x}\right)^{p + 1} \left[x^{b_1 + p} \mathcal{F}^{(1)}_1(a, b_1, b_2; c; t_1, t_2, k_1, k_2, x, y)\right]\nonumber\\
%	& =  \left(\frac{\partial}{\partial x}\right)^{p} \left[\frac{\partial}{\partial x} \{x^{b_1 + p} \mathcal{F}^{(1)}_1(a, b_1, b_2; c; t_1, t_2, k_1, k_2, x, y)\}\right]\nonumber\\
	& = \left(\frac{\partial}{\partial x}\right)^{p} [(b_1 + p) x^{b_1 + p - 1} \nonumber\\
	& \quad \times \sum_{m, n \geq 0} \frac{(a)_{m+n} \, (b_1)_m \, (b_2)_n \, (-1)^{m k_1} \, (-1)^{n k_2} \, (-t_1)_{mk_1} \, (-t_2)_{nk_2}}{ (c)_{m+n} \, m! \, n!} \, x^{m} \, y^n\nonumber\\
	& \quad + x^{b_1 + p} \sum_{m, n \geq 0} \frac{(a)_{m+n} \, (b_1)_m \, (b_2)_n \, (-1)^{m k_1} \, (-1)^{n k_2} \, (-t_1)_{mk_1} \, (-t_2)_{nk_2}}{ (c)_{m+n} \, m! \, n!} \, m \, x^{ m - 1} \, y^n]\nonumber\\
	& = \sum_{m, n \geq 0} (b_1 + p + m)  \frac{(a)_{m + n} \, (b_1)_m \, (b_2)_n \, (-1)^{m k_1} \, (-1)^{n k_2} \, (-t_1)_{mk_1} \, (-t_2)_{nk_2}}{ (c)_{m+n} \, m! \, n!}  \, y^n\nonumber\\
& \quad \times \left(\frac{\partial}{\partial x}\right)^{p} x^{b_1 + p + m - 1} \nonumber\\
% & = x^{b_1 - 1} \, (b_1)_{p + 1}\sum_{m, n \geq 0}   \frac{(a)_{m + n} \, (b_1 + p + 1)_m \, (b_2)_n \, (-1)^{m k_1} \, (-1)^{n k_2} \, (-t_1)_{mk_1} \, (-t_2)_{nk_2}}{ (c)_{m+n} \, m! \, n!}  \, y^n \, x^m\nonumber\\
 & = x^{b_1 - 1} \, (b_1)_{p + 1} \mathcal{F}^{(1)}_1(a, b_1 + p + 1, b_2; c; t_1, t_2, k_1, k_2, x, y).
\end{align}
Hence formula is valid for $r = p+ 1$. Therefore, it is true for each $r \in \mathbb{N}$. It completes the proof. Similarly, other differential formulae can be proved. 
%\textbf {The equation \eqref{4.14} can be proved using $\left(\frac{\partial}{\partial x}\right)^r \, x^{m + b_1 + r - 1} = (b_1 + m)_r \, x^{m + b_1 - 1}$ and $(b_1)_m \, (b_1 + m)_r = (b_1)_r \, (b_1 + r)_m$. }
\section{Finite and infinite summation formulas}
In this section, we establish some finite and infinite summation formulas in terms of discrete Appell function $\mathcal{F}^{(1)}_1$. We list them in the following theorem.
\begin{theorem}
	The following summation formulas hold:
	\begin{align}
& \mathcal{F}^{(1)}_1(a, b_1 + r, b_2; c; t_1, t_2, k_1, k_2, x, y)\nonumber\\
& = \sum_{s = 0}^{r} {r \choose s} \frac{(a)_s \, (-1)^{sk_1} \, (-t_1)_{sk_1}}{(c)_s} \, x^s \, \mathcal{F}^{(1)}_1(a + s, b_1 + s, b_2; c + s; t_1 - sk_1, t_2, k_1, k_2, x, y);\label{5.1}\\
& \mathcal{F}^{(1)}_1(a, b_1, b_2 + r; c; t_1, t_2, k_1, k_2, x, y)\nonumber\\
& = \sum_{s = 0}^{r} {r \choose s} \frac{(a)_s \, (-1)^{sk_2} \, (-t_2)_{sk_2}}{(c)_s} \, y^s \, \mathcal{F}^{(1)}_1(a + s, b_1, b_2 + s; c + s; t_1, t_2 - sk_2, k_1, k_2, x, y);\label{5.2}\\
	&\sum_{r = 0}^{\infty} \frac{(a)_r}{r !} \, z^r \, \mathcal{F}^{(1)}_1 (a + r, b_1, b_2; c ; t_1, t_2, k_1, k_2, x, y)\nonumber\\
& = (1 - z)^{-a} \, \mathcal{F}^{(1)}_1 \left(a, b_1, b_2; c; t_1, t_2, k_1, k_2, \frac{x}{1 - z}, \frac{y}{1 - z}\right);\label{56}
\\[5 pt]
&\sum_{r = 0}^{\infty} \frac{(b_1)_r}{r !} \, z^r \, \mathcal{F}^{(1)}_1 (a, b_1 + r, b_2; c ; t_1, t_2, k_1, k_2, x, y)\nonumber\\
& = (1 - z)^{-b_1} \, \mathcal{F}^{(1)}_1 \left(a, b_1, b_2; c; t_1, t_2, k_1, k_2, \frac{x}{1 - z}, y\right);
\\[5 pt]
&\sum_{r = 0}^{\infty} \frac{(b_2)_r}{r !} \, z^r \, \mathcal{F}^{(1)}_1 (a, b_1, b_2 + r; c ; t_1, t_2, k_1, k_2, x, y)\nonumber\\
& = (1 - z)^{-b_2} \, \mathcal{F}^{(1)}_1 \left(a, b_1, b_2; c; t_1, t_2, k_1, k_2, x, \frac{y}{1 - z}\right).
	\end{align}
\end{theorem}  
\begin{proof}
The Leibnitz rule, for differentiation of product of functions, allows us to write
\begin{align}
	& \left(\frac{\partial}{\partial x}\right)^r \left[x^{b_1 + r - 1} \, \mathcal{F}^{(1)}_1(a, b_1, b_2; c; t_1, t_2, k_1, k_2, x, y)\right]\nonumber\\
	& = \sum_{s = 0}^{r} {r \choose s} \left[\left(\frac{\partial}{\partial x}\right)^{r - s} \, x^{b_1 + r - 1}\right] \, \left[\left(\frac{\partial}{\partial x}\right)^s  \, \mathcal{F}^{(1)}_1(a, b_1, b_2; c; t_1, t_2, k_1, k_2, x, y)\right]\nonumber\\
	& = \sum_{s = 0}^{r} {r \choose s} (-1)^{r - s} \, (1 - b_1 + s)_{r - s} \, x^{b_1 + s - 1} \, \frac{(a)_s \, (b_1)_s \, (-1)^{sk_1} \, (-t_1)_{sk_1}}{(c)_s}\nonumber\\
	& \quad \times \mathcal{F}^{(1)}_1(a + s, b_1 + s, b_2; c + s; t_1 - sk_1, t_2, k_1, k_2, x, y) \nonumber\\
	& = \sum_{s = 0}^{r} {r \choose s} \, (-1)^{r - s} \frac{(-1)^s \, (1 - b_1 - r)_r}{(b_1)_s} \, x^{b_1 + s - 1} \, \frac{(a)_s \, (b_1)_s \, (-1)^{sk_1} \, (-t_1)_{sk_1}}{(c)_s}\nonumber\\
	& \quad \times \mathcal{F}^{(1)}_1(a + s, b_1 + s, b_2; c + s; t_1 - sk_1, t_2, k_1, k_2, x, y) \nonumber\\
	& = (b_1)_r \, \sum_{s = 0}^{r} {r \choose s} \,  x^{b_1 + s - 1} \, \frac{(a)_s  \, (-1)^{sk_1} \, (-t_1)_{sk_1}}{(c)_s}\nonumber\\
	& \quad \times \mathcal{F}^{(1)}_1(a + s, b_1 + s, b_2; c + s; t_1 - sk_1, t_2, k_1, k_2, x, y).
\end{align}
Taking into account Equation \eqref{4.14}, we have
\begin{align}
&	x^{b_1 - 1} \, (b_1)_r \, \mathcal{F}^{(1)}_1(a, b_1 + r, b_2; c; t_1, t_2, k_1, k_2, x, y)\nonumber\\
	& = (b_1)_r \, \sum_{s = 0}^{r} {r \choose s} \,  x^{b_1 + s - 1} \, \frac{(a)_s  \, (-1)^{sk_1} \, (-t_1)_{sk_1}}{(c)_s}\nonumber\\
	& \quad \times \mathcal{F}^{(1)}_1(a + s, b_1 + s, b_2; c + s; t_1 - sk_1, t_2, k_1, k_2, x, y).
\end{align}
This implies
	\begin{align}
	& \mathcal{F}^{(1)}_1(a, b_1 + r, b_2; c; t_1, t_2, k_1, k_2, x, y)\nonumber\\
	& = \sum_{s = 0}^{r} {r \choose s} \frac{(a)_s \, (-1)^{sk_1} \, (-t_1)_{sk_1}}{(c)_s} \, x^s \, \mathcal{F}^{(1)}_1(a + s, b_1 + s, b_2; c + s; t_1 - sk_1, t_2, k_1, k_2, x, y).
\end{align}
This completes the proof of \eqref{5.1}. A similar procedure may be applied for proving \eqref{5.2}.
To prove \eqref{56}, we start with its right hand side and get
\begin{align}
&(1 - z)^{-a} \, \mathcal{F}^{(1)}_1 \left(a, b_1, b_2; c; t_1, t_2, k_1, k_2, \frac{x}{1 - z}, \frac{y}{1 - z}\right) \nonumber\\
& = \sum_{m,n \geq 0} (1 - z)^{- (a + m + n)} \, \frac{(a)_{m+n} \, (b_1)_m \, (b_2)_n \, (-1)^{m k_1} \, (-1)^{n k_2} \, (-t_1)_{mk_1} \, (-t_2)_{nk_2}}{ (c)_{m+n} \, m! \, n!} \ x^m \, y^n.
\end{align}
Using the binomial theorem $(1 - z)^{- (a + m + n)} = \sum_{r = 0}^{\infty} \frac{(a + m + n)_r}{r !} \, z^r$ and the identity $(a)_{m + n + r} = (a)_{m + n} \, (a + m + n)_r = (a)_r \, (a + r)_{m + n}$, we have
\begin{align}
	&(1 - z)^{-a} \, \mathcal{F}^{(1)}_1 \left(a, b_1, b_2; c; t_1, t_2, k_1, k_2, \frac{x}{1 - z}, \frac{y}{1 - z}\right) \nonumber\\
	& = \sum_{m,n \geq 0}  \, \sum_{r = 0}^{\infty} \frac{(a)_r}{r !} \, z^r  \frac{(a + r)_{m+n} \, (b_1)_m \, (b_2)_n \, (-1)^{m k_1} \, (-1)^{n k_2} \, (-t_1)_{mk_1} \, (-t_2)_{nk_2}}{ (c)_{m+n} \, m! \, n!} \ x^m \, y^n\nonumber\\
	& = \sum_{r = 0}^{\infty} \frac{(a)_r}{r !} \, z^r \, \mathcal{F}^{(1)}_1 (a + r, b_1, b_2; c ; t_1, t_2, k_1, k_2, x, y).
\end{align}
Hence, we arrived at the equation \eqref{56}. Remaining results can be proved in a similar manner.
\end{proof}
\section{Recursion Formulae}
\begin{theorem} 
	The discrete Appell function $\mathcal{F}^{(1)}_1$ satisfies the following recursion formulas:
\begin{align}
& \mathcal{F}^{(1)}_1 (a + s, b_1, b_2; c ; t_1, t_2, k_1, k_2, x, y) \nonumber\\
& = \mathcal{F}^{(1)}_1 (a, b_1, b_2; c ; t_1, t_2, k_1, k_2, x, y)\nonumber\\
& \quad + \frac{(-1)^{k_1} \, (-t_1)_{k_1} \, b_1 \, x}{c} \sum_{r = 1}^{s} \mathcal{F}^{(1)}_1 (a + r, b_1 + 1, b_2; c + 1; t_1 - k_1, t_2, k_1, k_2, x, y) \nonumber\\
& \quad + \frac{(-1)^{k_2} \, (-t_2)_{k_2} \, b_2 \, y}{c}  \sum_{r = 1}^{s} \mathcal{F}^{(1)}_1 (a + r, b_1, b_2 + 1; c + 1; t_1, t_2 - k_2, k_1, k_2, x, y);\label{e6.1}\\
 & \mathcal{F}^{(1)}_1 (a - s, b_1, b_2; c ; t_1, t_2, k_1, k_2, x, y) \nonumber\\
& = \mathcal{F}^{(1)}_1 (a, b_1, b_2; c ; t_1, t_2, k_1, k_2, x, y) \nonumber\\
& \quad - \frac{(-1)^{k_1} \, (-t_1)_{k_1} \, b_1 \, x}{c}  \sum_{r = 0}^{s - 1} \mathcal{F}^{(1)}_1 (a - r, b_1 + 1, b_2; c + 1; t_1 - k_1, t_2, k_1, k_2, x, y)\nonumber\\
& \quad - \frac{(-1)^{k_2} \, (-t_2)_{k_2} \, b_2 \, y}{c} \sum_{r = 0}^{s - 1} \mathcal{F}^{(1)}_1 (a - r, b_1, b_2 + 1; c + 1; t_1, t_2 - k_2, k_1, k_2, x, y);\\
& \mathcal{F}^{(1)}_1 (a, b_1 + s, b_2; c ; t_1, t_2, k_1, k_2, x, y) \nonumber\\
& = \mathcal{F}^{(1)}_1 (a, b_1, b_2; c ; t_1, t_2, k_1, k_2, x, y) \nonumber\\
& \quad + \frac{(-1)^{k_1} \, (-t_1)_{k_1} \, a \, x}{c} \sum_{r = 1}^{s} \mathcal{F}^{(1)}_1 (a + 1, b_1 + r, b_2; c + 1; t_1 - k_1, t_2, k_1, k_2, x, y);\label{e63}\\
& \mathcal{F}^{(1)}_1 (a, b_1 - s, b_2; c ; t_1, t_2, k_1, k_2, x, y) \nonumber\\
& = \mathcal{F}^{(1)}_1 (a, b_1, b_2; c ; t_1, t_2, k_1, k_2, x, y) \nonumber\\
& \quad - \frac{(-1)^{k_1} \, (-t_1)_{k_1} \, a \, x}{c}  \sum_{r = 0}^{s - 1} \mathcal{F}^{(1)}_1 (a + 1, b_1 - r, b_2; c + 1; t_1 - k_1, t_2, k_1, k_2, x, y);\label{e64}\\
& \mathcal{F}^{(1)}_1 (a, b_1, b_2; c - s; t_1, t_2, k_1, k_2, x, y) \nonumber\\
& = \mathcal{F}^{(1)}_1 (a, b_1, b_2; c ; t_1, t_2, k_1, k_2, x, y) \nonumber\\
& \quad + (-1)^{k_1} \, (-t_1)_{k_1} \, a \, b_1 \, x  \sum_{r = 1}^{s} \frac{\mathcal{F}^{(1)}_1 (a + 1, b_1 + 1, b_2; c + 2 - r; t_1 - k_1, t_2, k_1, k_2, x, y)}{(c - r) \, (c - r + 1)} \nonumber\\
& \quad + (-1)^{k_2} \, (-t_2)_{k_2} \, a \, b_2 \, y  \sum_{r = 1}^{s} \frac{\mathcal{F}^{(1)}_1 (a + 1, b_1, b_2 + 1; c + 2 - r; t_1, t_2 - k_2, k_1, k_2, x, y)}{(c - r) \, (c - r + 1)}.
\end{align}
\end{theorem}
\begin{proof}
We first prove the formula \eqref{e6.1} for $s=1$:
\begin{align}
 &\mathcal{F}^{(1)}_1 (a, b_1, b_2; c ; t_1, t_2, k_1, k_2, x, y)\nonumber\\
 & \quad  + \frac{(-1)^{k_1} \, (-t_1)_{k_1} \, b_1 \, x}{c} \mathcal{F}^{(1)}_1 (a + 1, b_1 + 1, b_2; c + 1; t_1 - k_1, t_2, k_1, k_2, x, y) \nonumber\\
 & \quad  + \frac{(-1)^{k_2} \, (-t_2)_{k_2} \, b_2 \, y}{c}   \mathcal{F}^{(1)}_1 (a + 1, b_1, b_2 + 1; c + 1; t_1, t_2 - k_2, k_1, k_2, x, y)\nonumber\\
& = \sum_{m,n \geq 0}  \, \frac{(a)_{m+n} \, (b_1)_m \, (b_2)_n \, (-1)^{m k_1} \, (-1)^{n k_2} \, (-t_1)_{mk_1} \, (-t_2)_{nk_2}}{ (c)_{m+n} \, m! \, n!} \ x^m \, y^n \nonumber\\
& \quad  + \frac{(-1)^{k_1} \, (-t_1)_{k_1} \, b_1 \, x}{c} \nonumber\\
& \quad \times \sum_{m,n \geq 0}  \, \frac{(a + 1)_{m+n} \, (b_1 + 1)_m \, (b_2)_n \, (-1)^{m k_1} \, (-t_1 + k_1)_{mk_1} \, (-1)^{nk_2} \, (-t_2)_{nk_2}}{ (c + 1)_{m+n} \, m! \, n!} \ x^m \, y^n \nonumber\\
& \quad  + \frac{(-1)^{k_2} \, (-t_2)_{k_2} \, b_2 \, y}{c} \nonumber\\
& \quad \times  \sum_{m,n \geq 0}  \, \frac{(a + 1)_{m+n} \, (b_1)_m \, (b_2 + 1)_n \, (-1)^{m k_1} \, (-t_1)_{mk_1} \, (-1)^{n k_2} \, (-t_2 + k_2)_{nk_2}}{ (c + 1)_{m+n} \, m! \, n!} \ x^m \, y^n\nonumber\\
& = \sum_{m,n \geq 0}  \, \frac{(a)_{m+n} \, (b_1)_m \, (b_2)_n \, (-1)^{m k_1} \, (-1)^{n k_2} \, (-t_1)_{mk_1} \, (-t_2)_{nk_2}}{ (c)_{m+n} \, m! \, n!} \ x^m \, y^n \nonumber\\
& \quad + \sum_{m,n \geq 0} \frac{m}{a} \, \frac{(a)_{m+n} \, (b_1)_m \, (b_2)_n \, (-1)^{m k_1} \, (-1)^{n k_2} \, (-t_1)_{mk_1} \, (-t_2)_{nk_2}}{ (c)_{m+n} \, m! \, n!} \ x^m \, y^n \nonumber\\
& \quad + \sum_{m,n \geq 0}  \frac{n}{a} \, \frac{(a)_{m+n} \, (b_1)_m \, (b_2)_n \, (-1)^{m k_1} \, (-1)^{n k_2} \, (-t_1)_{mk_1} \, (-t_2)_{nk_2}}{ (c)_{m+n} \, m! \, n!} \ x^m \, y^n\nonumber\\
& =  \sum_{m,n \geq 0}  \frac{a + m + n}{a} \, \frac{(a)_{m+n} \, (b_1)_m \, (b_2)_n \, (-1)^{m k_1} \, (-1)^{n k_2} \, (-t_1)_{mk_1} \, (-t_2)_{nk_2}}{ (c)_{m+n} \, m! \, n!} \ x^m \, y^n\nonumber\\
& = \mathcal{F}^{(1)}_1 (a + 1, b_1, b_2; c; t_1, t_2, k_1, k_2, x, y). \label{6.6}
\end{align}
So the formula is true for $s = 1$. Assuming the validity of \eqref{e6.1} for $s = p$, we get 
%and hence we get
%\begin{align}
%	& \mathcal{F}^{(1)}_1 (a + p, b_1, b_2; c ; t_1, t_2, k_1, k_2, x, y) \nonumber\\
%	& = \mathcal{F}^{(1)}_1 (a, b_1, b_2; c ; t_1, t_2, k_1, k_2, x, y)\nonumber\\
%	& \quad + \frac{(-1)^{k_1} \, (-t_1)_{k_1} \, b_1 \, x}{c} \sum_{r = 1}^{p} \mathcal{F}^{(1)}_1 (a + r, b_1 + 1, b_2; c + 1; t_1 - k_1, t_2, k_1, k_2, x, y) \nonumber\\
%	& \quad + \frac{(-1)^{k_2} \, (-t_2)_{k_2} \, b_2 \, y}{c}  \sum_{r = 1}^{p} \mathcal{F}^{(1)}_1 (a + r, b_1, b_2 + 1; c + 1; t_1, t_2 - k_2, k_1, k_2, x, y).
%\end{align}
%Using this
\begin{align}
	& \mathcal{F}^{(1)}_1 (a, b_1, b_2; c ; t_1, t_2, k_1, k_2, x, y)\nonumber\\
& \quad + \frac{(-1)^{k_1} \, (-t_1)_{k_1} \, b_1 \, x}{c} \sum_{r = 1}^{p + 1} \mathcal{F}^{(1)}_1 (a + r, b_1 + 1, b_2; c + 1; t_1 - k_1, t_2, k_1, k_2, x, y) \nonumber\\
& \quad + \frac{(-1)^{k_2} \, (-t_2)_{k_2} \, b_2 \, y}{c}  \sum_{r = 1}^{p + 1} \mathcal{F}^{(1)}_1 (a + r, b_1, b_2 + 1; c + 1; t_1, t_2 - k_2, k_1, k_2, x, y)\nonumber\\
& = \mathcal{F}^{(1)}_1 (a + p, b_1, b_2; c ; t_1, t_2, k_1, k_2, x, y)\nonumber\\
& \quad + \frac{(-1)^{k_1} \, (-t_1)_{k_1} \, b_1 \, x}{c}  \mathcal{F}^{(1)}_1 (a + p + 1, b_1 + 1, b_2; c + 1; t_1 - k_1, t_2, k_1, k_2, x, y) \nonumber\\
& \quad + \frac{(-1)^{k_2} \, (-t_2)_{k_2} \, b_2 \, y}{c}  \mathcal{F}^{(1)}_1 (a + p + 1, b_1, b_2 + 1; c + 1; t_1, t_2 - k_2, k_1, k_2, x, y) \nonumber\\
%\label{6.8}
%\end{align}
%The Equations \eqref{6.6} and \eqref{6.8} together yield
%\begin{align}
%&\mathcal{F}^{(1)}_1 (a, b_1, b_2; c ; t_1, t_2, k_1, k_2, x, y)\nonumber\\
%	& \quad + \frac{(-1)^{k_1} \, (-t_1)_{k_1} \, b_1 \, x}{c} \sum_{r = 1}^{p + 1} \mathcal{F}^{(1)}_1 (a + r, b_1 + 1, b_2; c + 1; t_1 - k_1, t_2, k_1, k_2, x, y) \nonumber\\
%	& \quad + \frac{(-1)^{k_2} \, (-t_2)_{k_2} \, b_2 \, y}{c}  \sum_{r = 1}^{p + 1} \mathcal{F}^{(1)}_1 (a + r, b_1, b_2 + 1; c + 1; t_1, t_2 - k_2, k_1, k_2, x, y)\nonumber\\
	& = \mathcal{F}^{(1)}_1 (a + p + 1, b_1, b_2; c ; t_1, t_2, k_1, k_2, x, y).
\end{align}
Hence, \eqref{e6.1} is true for $s = p + 1$ and therefore by induction is valid for all $s \in \mathbb{N}$. The other formulae can be verified in the same manner. 

Note that the recursion formulas for $\mathcal{F}^{(1)}_1 (a, b_1, b_2 \pm s; c ; t_1, t_2, k_1, k_2, x, y)$ can be obtained by interchanging $b_1$ and $b_2$ in \eqref{e63} and \eqref{e64}, respectively. 
\end{proof}
Mullen \cite{mj} obtained some differential recursion formulae for Appell functions. Following the same approach, we produce here a complete list of difference and differential recursion formulae satisfied by discrete Appell function $F_1^{(1)}$. We start our findings with the differential recursion formulae first. To obtain them, we list simple differential relations as
\begin{align}
&	a \, \mathcal{F}^{(1)}_1 (a + 1)  = (a + \theta + \phi) \, \mathcal{F}^{(1)}_1;\\
& (a + \theta + \phi - 1) \, \mathcal{F}^{(1)}_1 (a - 1)  =	(a - 1) \, \mathcal{F}^{(1)}_1;\\
&	b_1 \, \mathcal{F}^{(1)}_1 (b_1 + 1)  = (b_1 + \theta) \, \mathcal{F}^{(1)}_1;\\
&(b_1 + \theta - 1) \, \mathcal{F}^{(1)}_1 (b_1 - 1)  =	(b_1 - 1) \, \mathcal{F}^{(1)}_1;\\
&	b_2 \, \mathcal{F}^{(1)}_1 (b_2 + 1)  = (b_2 + \phi) \, \mathcal{F}^{(1)}_1;\\
& (b_2  + \phi - 1) \, \mathcal{F}^{(1)}_1 (b_2 - 1)  =	(b_2 - 1) \, \mathcal{F}^{(1)}_1;\\
&	(c - 1) \, \mathcal{F}^{(1)}_1 (c - 1)  = (c + \theta + \phi - 1) \, \mathcal{F}^{(1)}_1;\\
& (c + \theta + \phi) \, \mathcal{F}^{(1)}_1 (c + 1)  =	c \, \mathcal{F}^{(1)}_1,
\end{align}
where
\begin{align}
\mathcal{F}^{(1)}_1 & = \mathcal{F}^{(1)}_1 (a, b_1, b_2; c; t_1, t_2, k_1, k_2, x, y);\nonumber\\
\mathcal{F}^{(1)}_1 (a + 1) & = \mathcal{F}^{(1)}_1 (a + 1, b_1, b_2; c; t_1, t_2, k_1, k_2, x, y);\nonumber\\
\mathcal{F}^{(1)}_1 (a - 1) & = \mathcal{F}^{(1)}_1 (a - 1, b_1, b_2; c; t_1, t_2, k_1, k_2, x, y)\nonumber
\end{align}
and so on. On combining any two of the above relations, we get a first order or second order differential recursion relation. For example, if  $\mathcal{F}^{(1)}_1 (a + 1)$ is combined with $\mathcal{F}^{(1)}_1 (a - 1)$ or $\mathcal{F}^{(1)}_1 (b_1 - 1)$ or $\mathcal{F}^{(1)}_1 (b_2 - 1)$ or $\mathcal{F}^{(1)}_1 (c + 1)$, we get a second order differential recursion relation as
\begin{align}
& a \, (a - 1) \, \mathcal{F}^{(1)}_1 (a + 1) - (a + \theta + \phi) \, (a + \theta + \phi - 1) \mathcal{F}^{(1)}_1 (a - 1) = 0;\\
& a \, (b_1 - 1) \, \mathcal{F}^{(1)}_1 (a + 1) - (a + \theta + \phi) \, (b_1 + \theta  - 1) \mathcal{F}^{(1)}_1 (b_1 - 1) = 0;\\
& a \, (b_2 - 1) \, \mathcal{F}^{(1)}_1 (a + 1) - (a + \theta + \phi) \, (b_2 + \phi - 1) \mathcal{F}^{(1)}_1 (b_2 - 1) = 0;\\
& a \, c \, \mathcal{F}^{(1)}_1 (a + 1) - (a + \theta + \phi) \, (c + \theta + \phi) \mathcal{F}^{(1)}_1 (c + 1) = 0.
\end{align}
In the same way, if $\mathcal{F}^{(1)}_1 (a + 1)$ is combined with $\mathcal{F}^{(1)}_1 (b_1 + 1)$ or $\mathcal{F}^{(1)}_1 (b_2 + 1)$ or $\mathcal{F}^{(1)}_1 (c - 1)$, we will get a first order differential recursion relations as given below
\begin{align}
	& a \, (b_1 + \theta) \, \mathcal{F}^{(1)}_1 (a + 1) - b_1 \, (a + \theta + \phi) \, \mathcal{F}^{(1)}_1 (b_1 + 1) = 0;\\
		& a \, (b_2 + \phi) \, \mathcal{F}^{(1)}_1 (a + 1) - b_2 \, (a + \theta + \phi) \, \mathcal{F}^{(1)}_1 (b_2 + 1) = 0;\\
	& a \, (c + \theta + \phi - 1) \, \mathcal{F}^{(1)}_1 (a + 1) - (c - 1) \, (a + \theta + \phi) \, \mathcal{F}^{(1)}_1 (c - 1) = 0.
\end{align}
We apply this idea of obtaining differential recursion relations. $\mathcal{F}^{(1)}_1 (a - 1)$ will give six distinct recursion relations; $\mathcal{F}^{(1)}_1 (b_1 + 1)$ gives five such relations; $\mathcal{F}^{(1)}_1 (b_1 - 1)$ gives four; $\mathcal{F}^{(1)}_1 (b_2 + 1)$ three; $\mathcal{F}^{(1)}_1 (b_2 - 1)$ two and $\mathcal{F}^{(1)}_1 (c - 1)$ only one relation. In this way, we get a complete list of 28 relations. Among these 28, we have already presented seven relations of $\mathcal{F}^{(1)}_1 (a + 1)$ and others. We now list the rest of 21 relations.
 \begin{align}
 	&(a + \theta + \phi - 1) \, (b_1 + \theta) \, \mathcal{F}^{(1)}_1 (a - 1) - b_1 \, (a - 1)  \, \mathcal{F}^{(1)}_1 (b_1 + 1) = 0;\\
 &	(a + \theta + \phi - 1) \, (b_2 + \phi) \, \mathcal{F}^{(1)}_1 (a - 1) - b_2 \, (a - 1)  \, \mathcal{F}^{(1)}_1 (b_2 + 1) = 0;\\
 &	(a + \theta + \phi - 1) \, (c + \theta + \phi - 1) \, \mathcal{F}^{(1)}_1 (a - 1) - (c - 1) \, (a - 1)  \, \mathcal{F}^{(1)}_1 (c - 1) = 0;\\
 & (b_1 - 1) \,	(a + \theta + \phi - 1) \,  \mathcal{F}^{(1)}_1 (a - 1) -  (a - 1)  \, (b_1 + \theta - 1) \mathcal{F}^{(1)}_1 (b_1 - 1) = 0;\\
 & (b_2 - 1) \,	(a + \theta + \phi - 1) \,  \mathcal{F}^{(1)}_1 (a - 1) -  (a - 1)  \, (b_2 + \phi - 1) \mathcal{F}^{(1)}_1 (b_2 - 1) = 0;\\
 & c \,	(a + \theta + \phi - 1) \,  \mathcal{F}^{(1)}_1 (a - 1) -  (a - 1)  \, (c + \theta + \phi) \mathcal{F}^{(1)}_1 (c + 1) = 0;\\
 & b_1 \,	(b_1 - 1) \,  \mathcal{F}^{(1)}_1 (b_1 + 1) -  (b_1 + \theta)  \, (b_1 + \theta - 1) \mathcal{F}^{(1)}_1 (b_1 - 1) = 0;\\
 & b_1 \,	(b_2 + \phi) \,  \mathcal{F}^{(1)}_1 (b_1 + 1) -  b_2  \, (b_1 + \theta ) \mathcal{F}^{(1)}_1 (b_2 + 1) = 0;\\
  & b_1 \,	(b_2 - 1) \,  \mathcal{F}^{(1)}_1 (b_1 + 1) -  (b_1 + \theta)  \, (b_2 + \phi - 1) \mathcal{F}^{(1)}_1 (b_2 - 1) = 0;\\
  & b_1 \,	(c + \theta + \phi - 1) \,  \mathcal{F}^{(1)}_1 (b_1 + 1) -  (c - 1)  \, (b_1 + \theta) \mathcal{F}^{(1)}_1 (c - 1) = 0;\\
  & b_1 \,	c \,  \mathcal{F}^{(1)}_1 (b_1 + 1) -  (c + \theta +\phi)  \, (b_1 + \theta) \mathcal{F}^{(1)}_1 (c + 1) = 0;\\
   & b_2 \,	(b_1 - 1) \,  \mathcal{F}^{(1)}_1 (b_2 + 1) -  (b_2 + \phi)  \, (b_1 + \theta - 1) \mathcal{F}^{(1)}_1 (b_1 - 1) = 0;\\
  & b_2 \,	(b_2 - 1) \,  \mathcal{F}^{(1)}_1 (b_2 + 1) -  (b_2 + \phi)  \, (b_2 + \phi - 1) \mathcal{F}^{(1)}_1 (b_2 - 1) = 0;\\
  & b_2 \,	(c + \theta + \phi - 1) \,  \mathcal{F}^{(1)}_1 (b_2 + 1) -  (c - 1)  \, (b_2 + \phi) \mathcal{F}^{(1)}_1 (c - 1) = 0;\\
  & b_2 \,	c \,  \mathcal{F}^{(1)}_1 (b_2 + 1) -  (c + \theta +\phi)  \, (b_2 + \phi) \mathcal{F}^{(1)}_1 (c + 1) = 0;\\
  & (b_2 - 1) \,	(b_1 + \theta - 1) \,  \mathcal{F}^{(1)}_1 (b_1 - 1) -  (b_1 - 1)  \, (b_2 + \phi - 1) \mathcal{F}^{(1)}_1 (b_2 - 1) = 0;\\
  & (b_1 + \theta - 1) \,	(c + \theta + \phi - 1) \,  \mathcal{F}^{(1)}_1 (b_1 - 1) -  (c - 1)  \, (b_1 - 1) \mathcal{F}^{(1)}_1 (c - 1) = 0;\\
  & 	c \, (b_1 + \theta - 1) \, \mathcal{F}^{(1)}_1 (b_1 - 1) - (b_1 - 1)  (c + \theta +\phi)  \,  \mathcal{F}^{(1)}_1 (c + 1) = 0;\\
  & (b_2 + \phi - 1) \,	(c + \theta + \phi - 1) \,  \mathcal{F}^{(1)}_1 (b_2 - 1) -  (c - 1)  \, (b_2 - 1) \mathcal{F}^{(1)}_1 (c - 1) = 0;\\
  & 	c \, (b_2 + \phi - 1) \, \mathcal{F}^{(1)}_1 (b_2 - 1) - (b_2 - 1)  (c + \theta +\phi)  \,  \mathcal{F}^{(1)}_1 (c + 1) = 0;\\
   & 	c \, (c - 1) \, \mathcal{F}^{(1)}_1 (c - 1) - (c + \theta + \phi - 1)  (c + \theta +\phi)  \,  \mathcal{F}^{(1)}_1 (c + 1) = 0.
 \end{align}
Similarly using  the difference relations: 
\begin{align}
	&	a \, \mathcal{F}^{(1)}_1 (a + 1)  = \left(a + \frac{1}{k_1} \Theta_{t_1} + \frac{1}{k_2}\Theta_{t_2}\right) \, \mathcal{F}^{(1)}_1;\\
	& \left(a + \frac{1}{k_1} \Theta_{t_1} + \frac{1}{k_2}\Theta_{t_2} - 1\right) \, \mathcal{F}^{(1)}_1 (a - 1)  =	(a - 1) \, \mathcal{F}^{(1)}_1;\\
	&	b_1 \, \mathcal{F}^{(1)}_1 (b_1 + 1)  = \left(b_1 + \frac{1}{k_1} \Theta_{t_1}\right) \, \mathcal{F}^{(1)}_1;\\
	&\left(b_1 + \frac{1}{k_1} \Theta_{t_1} - 1\right) \, \mathcal{F}^{(1)}_1 (b_1 - 1)  =	(b_1 - 1) \, \mathcal{F}^{(1)}_1;\\
	&	b_2 \, \mathcal{F}^{(1)}_1 (b_2 + 1)  = \left(b_2 + \frac{1}{k_2} \, \Theta_{t_2}\right) \, \mathcal{F}^{(1)}_1;\\
	& \left(b_2 + \frac{1}{k_2} \, \Theta_{t_2} - 1\right) \, \mathcal{F}^{(1)}_1 (b_2 - 1)  =	(b_2 - 1) \, \mathcal{F}^{(1)}_1;\\
	&	(c - 1) \, \mathcal{F}^{(1)}_1 (c - 1)  = \left(c + \frac{1}{k_1} \Theta_{t_1} + \frac{1}{k_2}\Theta_{t_2} - 1\right) \, \mathcal{F}^{(1)}_1;\\
	& \left(c + \frac{1}{k_1} \Theta_{t_1} + \frac{1}{k_2}\Theta_{t_2}\right) \, \mathcal{F}^{(1)}_1 (c + 1)  =	c \, \mathcal{F}^{(1)}_1,
\end{align}
we will get following 28 difference recursion relations obeyed by discrete Appell function $\mathcal{F}^{(1)}_1$. For brevity, the proofs are omitted.
\begin{align}
	& a \, (a - 1) \, \mathcal{F}^{(1)}_1 (a + 1)\nonumber\\
	& \quad - \left(a + \frac{1}{k_1} \Theta_{t_1} + \frac{1}{k_2}\Theta_{t_2}\right) \, \left(a + \frac{1}{k_1} \Theta_{t_1} + \frac{1}{k_2}\Theta_{t_2} - 1\right) \mathcal{F}^{(1)}_1 (a - 1) = 0;\\
	& a \, (b_1 - 1) \, \mathcal{F}^{(1)}_1 (a + 1) \nonumber\\
	& \quad - \left(a + \frac{1}{k_1} \Theta_{t_1} + \frac{1}{k_2}\Theta_{t_2} \right) \, \left(b_1 + \frac{1}{k_1} \Theta_{t_1}  - 1\right) \mathcal{F}^{(1)}_1 (b_1 - 1) = 0;\\
	& a \, (b_2 - 1) \, \mathcal{F}^{(1)}_1 (a + 1)\nonumber\\
	& \quad - \left(a + \frac{1}{k_1} \Theta_{t_1} + \frac{1}{k_2}\Theta_{t_2}\right) \, \left(b_2 + \frac{1}{k_2} \Theta_{t_2} - 1\right) \mathcal{F}^{(1)}_1 (b_2 - 1) = 0;\\
	& a \, c \, \mathcal{F}^{(1)}_1 (a + 1)\nonumber\\
&	\quad  - \left(a + \frac{1}{k_1} \Theta_{t_1} + \frac{1}{k_2}\Theta_{t_2}\right) \, \left(c + \frac{1}{k_1} \Theta_{t_1} + \frac{1}{k_2}\Theta_{t_2} \right) \mathcal{F}^{(1)}_1 (c + 1) = 0;\\
	& a \, \left(b_1 + \frac{1}{k_1} \Theta_{t_1}\right) \, \mathcal{F}^{(1)}_1 (a + 1) - b_1 \, \left(a + \frac{1}{k_1} \Theta_{t_1} + \frac{1}{k_2}\Theta_{t_2}\right) \, \mathcal{F}^{(1)}_1 (b_1 + 1) = 0;\\
	& a \, \left(b_2 + \frac{1}{k_2} \Theta_{t_2}\right) \, \mathcal{F}^{(1)}_1 (a + 1) - b_2 \, \left(a + \frac{1}{k_1} \Theta_{t_1} + \frac{1}{k_2}\Theta_{t_2}\right) \, \mathcal{F}^{(1)}_1 (b_2 + 1) = 0;\\
	& a \, \left(c + \frac{1}{k_1} \Theta_{t_1} + \frac{1}{k_2}\Theta_{t_2} - 1\right) \, \mathcal{F}^{(1)}_1 (a + 1)\nonumber\\
	& \quad  - (c - 1) \, \left(a + \frac{1}{k_1} \Theta_{t_1} + \frac{1}{k_2}\Theta_{t_2}\right) \, \mathcal{F}^{(1)}_1 (c - 1) = 0;\\
	&\left(a + \frac{1}{k_1} \Theta_{t_1} + \frac{1}{k_2}\Theta_{t_2} - 1\right) \, \left(b_1 + \frac{1}{k_1} \Theta_{t_1} \right) \, \mathcal{F}^{(1)}_1 (a - 1) - b_1 \, (a - 1)  \, \mathcal{F}^{(1)}_1 (b_1 + 1) = 0;\\
	&	\left(a + \frac{1}{k_1} \Theta_{t_1} + \frac{1}{k_2}\Theta_{t_2} - 1\right) \, \left(b_2 + \frac{1}{k_2}  \Theta_{t_2}\right) \, \mathcal{F}^{(1)}_1 (a - 1) - b_2 \, (a - 1)  \, \mathcal{F}^{(1)}_1 (b_2 + 1) = 0;\\
	&	\left(a + \frac{1}{k_1} \Theta_{t_1} + \frac{1}{k_2}\Theta_{t_2} - 1\right) \, \left(c + \frac{1}{k_1} \Theta_{t_1} + \frac{1}{k_2}\Theta_{t_2} - 1\right) \, \mathcal{F}^{(1)}_1 (a - 1)\nonumber\\
	& \quad  - (c - 1) \, (a - 1)  \, \mathcal{F}^{(1)}_1 (c - 1) = 0;\\
	& (b_1 - 1) \,	\left(a + \frac{1}{k_1} \Theta_{t_1} + \frac{1}{k_2}\Theta_{t_2} - 1\right) \,  \mathcal{F}^{(1)}_1 (a - 1) \nonumber\\
	& \quad -  (a - 1)  \, \left(b_1 + \frac{1}{k_1} \Theta_{t_1}  - 1\right) \mathcal{F}^{(1)}_1 (b_1 - 1) = 0;\\
	& (b_2 - 1) \,	\left(a + \frac{1}{k_1} \Theta_{t_1} + \frac{1}{k_2}\Theta_{t_2} - 1\right) \,  \mathcal{F}^{(1)}_1 (a - 1)\nonumber\\
	& \quad  -  (a - 1)  \, \left(b_2 + \frac{1}{k_2} \Theta_{t_2} - 1\right) \mathcal{F}^{(1)}_1 (b_2 - 1) = 0;\\
	& c \,	\left(a + \frac{1}{k_1} \Theta_{t_1} + \frac{1}{k_2}\Theta_{t_2} - 1\right) \,  \mathcal{F}^{(1)}_1 (a - 1)\nonumber\\
	& \quad -  (a - 1)  \, \left(c + \frac{1}{k_1} \Theta_{t_1} + \frac{1}{k_2}\Theta_{t_2}\right) \mathcal{F}^{(1)}_1 (c + 1) = 0;\\
	& b_1 \,	(b_1 - 1) \,  \mathcal{F}^{(1)}_1 (b_1 + 1) -  \left(b_1 + \frac{1}{k_1} \Theta_{t_1} \right)  \, \left(b_1 + \frac{1}{k_1} \Theta_{t_1}  - 1\right) \mathcal{F}^{(1)}_1 (b_1 - 1) = 0;\\
	& b_1 \,	\left(b_2 + \frac{1}{k_2} \Theta_{t_2}\right) \,  \mathcal{F}^{(1)}_1 (b_1 + 1) -  b_2  \, \left(b_1 + \frac{1}{k_1} \Theta_{t_1}\right) \mathcal{F}^{(1)}_1 (b_2 + 1) = 0;\\
	& b_1 \,	(b_2 - 1) \,  \mathcal{F}^{(1)}_1 (b_1 + 1) -  \left(b_1 + \frac{1}{k_1} \Theta_{t_1}\right)  \, \left(b_2 + \frac{1}{k_2} \Theta_{t_2} - 1\right) \mathcal{F}^{(1)}_1 (b_2 - 1) = 0;\\
	& b_1 \,	\left(c + \frac{1}{k_1} \Theta_{t_1} + \frac{1}{k_2}\Theta_{t_2} -  1\right) \,  \mathcal{F}^{(1)}_1 (b_1 + 1) -  (c - 1)  \, \left(b_1 + \frac{1}{k_1} \Theta_{t_1} \right) \mathcal{F}^{(1)}_1 (c - 1) = 0;\\
	& b_1 \,	c \,  \mathcal{F}^{(1)}_1 (b_1 + 1) -  \left(c + \frac{1}{k_1} \Theta_{t_1} + \frac{1}{k_2}\Theta_{t_2}\right)  \, \left(b_1 + \frac{1}{k_1} \Theta_{t_1}\right) \mathcal{F}^{(1)}_1 (c + 1) = 0;\\
	& b_2 \,	(b_1 - 1) \,  \mathcal{F}^{(1)}_1 (b_2 + 1) -  \left(b_2 + \frac{1}{k_2} \Theta_{t_2}\right)  \, \left(b_1 + \frac{1}{k_1} \Theta_{t_1}  - 1\right) \mathcal{F}^{(1)}_1 (b_1 - 1) = 0;\\
	& b_2 \,	(b_2 - 1) \,  \mathcal{F}^{(1)}_1 (b_2 + 1) -  \left(b_2 + \frac{1}{k_2} \Theta_{t_2}\right)  \, \left(b_2 + \frac{1}{k_2} \Theta_{t_2}  - 1\right) \mathcal{F}^{(1)}_1 (b_2 - 1) = 0;\\
	& b_2 \,	\left(c + \frac{1}{k_1} \Theta_{t_1} + \frac{1}{k_2}\Theta_{t_2} - 1\right) \,  \mathcal{F}^{(1)}_1 (b_2 + 1) -  (c - 1)  \, \left(b_2 + \frac{1}{k_2}  \Theta_{t_2}\right) \mathcal{F}^{(1)}_1 (c - 1) = 0;\\
	& b_2 \,	c \,  \mathcal{F}^{(1)}_1 (b_2 + 1) -  \left(c +\frac{1}{k_1} \Theta_{t_1} + \frac{1}{k_2}\Theta_{t_2}\right)  \, \left(b_2 + \frac{1}{k_2} \Theta_{t_2}\right) \mathcal{F}^{(1)}_1 (c + 1) = 0;\\
	& (b_2 - 1) \,	\left(b_1 + \frac{1}{k_1} \Theta_{t_1}  - 1\right) \,  \mathcal{F}^{(1)}_1 (b_1 - 1) -  (b_1 - 1)  \, \left(b_2 + \frac{1}{k_2}  \Theta_{t_2} - 1\right) \mathcal{F}^{(1)}_1 (b_2 - 1) = 0;\\
	& \left(b_1 + \frac{1}{k_1} \Theta_{t_1} - 1\right) \,	\left(c + \frac{1}{k_1} \Theta_{t_1} + \frac{1}{k_2}\Theta_{t_2} - 1\right) \,  \mathcal{F}^{(1)}_1 (b_1 - 1)\nonumber\\
	& \quad -  (c - 1)  \, (b_1 - 1) \mathcal{F}^{(1)}_1 (c - 1) = 0;\\
	& 	c \, \left(b_1 + \frac{1}{k_1} \Theta_{t_1} - 1\right) \, \mathcal{F}^{(1)}_1 (b_1 - 1) - (b_1 - 1)  \left(c + \frac{1}{k_1} \Theta_{t_1} + \frac{1}{k_2}\Theta_{t_2}\right)  \,  \mathcal{F}^{(1)}_1 (c + 1) = 0;\\
	& \left(b_2 + \frac{1}{k_2} \Theta_{t_2} - 1\right) \,	\left(c + \frac{1}{k_1} \Theta_{t_1} + \frac{1}{k_2}\Theta_{t_2} - 1\right) \,  \mathcal{F}^{(1)}_1 (b_2 - 1)\nonumber\\
	& \quad -  (c - 1)  \, (b_2 - 1) \mathcal{F}^{(1)}_1 (c - 1) = 0;\\
	& 	c \, \left(b_2 + \frac{1}{k_2}  \Theta_{t_2} - 1\right) \, \mathcal{F}^{(1)}_1 (b_2 - 1) - (b_2 - 1)  \left(c + \frac{1}{k_1} \Theta_{t_1} + \frac{1}{k_2}\Theta_{t_2}\right)  \,  \mathcal{F}^{(1)}_1 (c + 1) = 0;\\
	& 	c \, (c - 1) \, \mathcal{F}^{(1)}_1 (c - 1) - \left(c + \frac{1}{k_1} \Theta_{t_1} + \frac{1}{k_2}\Theta_{t_2} - 1\right)  \left(c + \frac{1}{k_1} \Theta_{t_1} + \frac{1}{k_2}\Theta_{t_2}\right)  \,  \mathcal{F}^{(1)}_1 (c + 1) = 0.
\end{align}

\section{The discrete Appell function $\mathcal{F}^{(2)}_1$}
In this section, we list the definitions, theorems and identities for the second discrete %analogue of 
Appell function $\mathcal{F}^{(2)}_1$. The proofs for these are similar to the one given for discrete Appell function $\mathcal{F}^{(1)}_1$ and so are omitted. We start with the particular cases. 

For $k = 0$, the discrete Appell function $\mathcal{F}^{(2)}_1$ defined in \eqref{3.2} reduces to Appell function $F_1$ and for $k = 1$, $\mathcal{F}^{(2)}_1$ coincides with the generalized Kamp\'e de F\'eriet function as:
\begin{align}
& \mathcal{F}^{(2)}_1(a, b_1, b_2; c; t, 1, x, y)
%	& = \sum_{m, n \geq 0} \frac{(a)_{m+n} \, (b_1)_m \, (b_2)_n \, (-1)^{(m + n)} \, (-t)_{m + n}}{ (c)_{m+n} \, m! \, n!} \ x^m \, y^n\nonumber\\
	 = F_{{1}:{0};{0}} ^{{2}:{1};{1}}\left(\begin{array}{ccc}
		a, -t: & b_1; & b_2\\
		c: & - ; & - 
	\end{array}; -x, \ -y\right).
\end{align}
The limiting cases of discrete Appell function $\mathcal{F}^{(2)}_1$ produce the second discrete form of the  Humbert functions $\phi_1$, $\phi_2$ and $\phi_3$. We denote these functions by $\phi^{(2)}_1$, $\phi^{(2)}_2$, $\phi^{(2)}_3$, defined as: 
\begin{align}
	\phi^{(2)}_1 \left(a, b_1; c; t, k, x,  y\right) &= \sum_{m,n \geq 0} \frac{(a)_{m+n} \, (b_1)_m \,  \, (-1)^{(m + n) k} \, (-t)_{(m + n)k}}{ (c)_{m+n} \, m! \, n!} \ x^m \, y^n;
	\\[5pt]
	\phi^{(2)}_2 \left( b_1, b_2; c; t, k, x,  y\right) & = \sum_{m,n \geq 0} \frac{(b_1)_{m} \, (b_2)_n \,  \, (-1)^{(m + n) k} \, (-t)_{(m + n)k}}{ (c)_{m+n} \, m! \, n!} \ x^m \, y^n
	\\[5pt]
	\phi^{(2)}_3 \left( b_1; c; t, k, x,  y\right) & = 	\sum_{m,n \geq 0} \frac{(b_1)_{m} \,  \, (-1)^{(m + n) k} \, (-t)_{(m + n)k}}{ (c)_{m+n} \, m! \, n!} \ x^m \, y^n.
\end{align}
The following degeneration relations between  discrete Appell function $\mathcal{F}^{(2)}_1$ and discrete Humbert functions $\phi^{(2)}_1$, $\phi^{(2)}_2$ and $\phi^{(2)}_3$ can be verified easily.
\begin{align}
	& \lim_{\varepsilon \to 0}  \mathcal{F}^{(2)}_1 \left(a, b_1, \frac{1}{\varepsilon}; c; t, k, x, \varepsilon \, y\right)  = \phi^{(2)}_1 \left(a, b_1; c; t, k, x,  y\right);\nonumber
	\\[5pt]
	&\lim_{\varepsilon \to 0}  \mathcal{F}^{(2)}_1 \left(\frac{1}{\varepsilon}, b_1, b_2 ; c; t, k, \varepsilon \, x, \varepsilon \, y\right) = \phi^{(2)}_2 \left( b_1, b_2; c; t, k, x,  y\right);\nonumber
	\\[5pt]
	&\lim_{\varepsilon \to 0}  \mathcal{F}^{(2)}_1 \left(\frac{1}{\varepsilon}, b_1,\frac{1}{\varepsilon} ; c; t, k, \varepsilon \, x, \varepsilon^2 \, y\right) = \phi^{(2)}_3 \left( b_1; c; t, k, x,  y\right).
\end{align}
The difference-differential equations satisfied by discrete function $\mathcal{F}^{(2)}_1$ are
\begin{align}
	&\left[\theta \left(\frac{1}{k} \Theta_{t} + c - 1\right) - k \, (-1)^k \, (-t)_k \, x \, \rho_t^k \,  \left(\frac{1}{k} \Theta_{t}  + a\right) \, \left(b_1 + \theta\right) \right] \mathcal{F}^{(2)}_1 = 0;\label{e7.6}\\
	&\left[\phi \left(\frac{1}{k} \Theta_{t} + c - 1\right) - k \, (-1)^k \, (-t)_k \, y \, \rho_t^k \,  \left(\frac{1}{k} \Theta_{t}  + a\right) \, \left(b_2 + \phi\right) \right] \mathcal{F}^{(2)}_1 = 0.\label{e7.7}
\end{align}	    
This implies
%\begin{align}
%	\frac{x \, (b_1 + \theta)}{\theta} \,  \mathcal{F}^{(2)}_1 = \frac{y \, (b_2 + \phi)}{\phi} \,  \mathcal{F}^{(2)}_1
%\end{align} 
%or
\begin{align}
	\left[ y \, \theta (b_2 + \phi) -  x \, \phi \left(b_1 + \theta  \right)\right] \,  \mathcal{F}^{(2)}_1 = 0.\label{e7.8} 
\end{align}
Observe that the difference-differential equations \eqref{e7.6} and \eqref{e7.7} satisfied by the discrete Appell  function $\mathcal{F}^{(2)}_1$  lead to a differential equation \eqref{e7.8} obeyed by $\mathcal{F}^{(2)}_1$. 
\subsection{Integral representations}
The integral representations of the discrete Appell  function $\mathcal{F}^{(2)}_1$ are given by:
	\begin{align}
		& \mathcal{F}^{(2)}_1(a, b_1, b_2; c; t, k, x, y)\nonumber\\
		& =\Gamma \left(\begin{array}{c}
			c\\
			a, c - a
		\end{array}\right) \int_{0}^{1} u^{a - 1} (1 - u)^{c - a - 1}\nonumber\\
	& \quad \times  F_{{0}:{0}; {0}} ^{{k}:{1}; {1}}\left(\begin{array}{ccc}
		\frac{-t}{k}, \dots, \frac{- t + k - 1}{k}: & b_1; & b_2\\
		-: & - ; & - 
	\end{array}; (-k)^k \, u x,  (-k)^k \, u y\right) du;\\
		& =\Gamma \left(\begin{array}{c}
			c\\
			b_1, b_2, c - b_1 - b_2
		\end{array}\right) \iint u^{b_1 - 1} v^{b_2 - 1} (1-u-v)^{c -  b_1 - b_2 - 1} \nonumber\\
		& \quad \times F_{{0}:{0}; {0}} ^{{k+1}:{0}; {0}}\left(\begin{array}{ccc}
			a, \frac{- t}{k}, \dots, \frac{- t + k - 1}{k}: & -; & -\\
			-: & - ; & - 
		\end{array}; (-k)^k \, u x,  (-k)^k \, v y\right) du\,dv,\nonumber\\
		&\qquad \qquad \qquad \qquad
		u\geq 0, \ v\geq 0, \ 1-u-v\geq 0;\\
& = \frac{1}{\Gamma (a)} \int_{0}^{\infty} e^{-u} \, u^{a - 1} \nonumber\\
	& \quad \times F_{{1}:{0}; {0}} ^{{k}:{1}; {1}}\left(\begin{array}{ccc}
		\frac{- t}{k}, \dots, \frac{- t + k - 1}{k} : & b_1 ; & b_2\\
		c: & - ; & - 
	\end{array}; (-k)^k \, u x,  (-k)^k \, u y\right) du;\\
	& = \frac{1}{\Gamma (b_1)} \int_{0}^{\infty} e^{-u} \, u^{b_1 - 1} \nonumber\\
	& \quad \times F_{{1}:{0}; {0}} ^{{k + 1}:{0}; {1}}\left(\begin{array}{ccc}
		a, \frac{- t}{k}, \dots, \frac{- t + k - 1}{k}  : & -; & b_2\\
		c: & - ; & - 
	\end{array}; (-k)^k \, u x,  (-k)^k \, y\right) du;\\
	& = \frac{1}{\Gamma (b_2)} \int_{0}^{\infty} e^{-v} \, v^{b_2 - 1} \nonumber\\
	& \quad \times F_{{1}:{0}; {0}} ^{{k + 1}:{1}; {0}}\left(\begin{array}{ccc}
		a, \frac{- t}{k}, \dots, \frac{- t + k - 1}{k}: & b_1; & -\\
		c: & - ; & - 
	\end{array}; (-k)^k \, x,  (-k)^k \, v y\right) dv\\
& = \frac{1}{\Gamma (-t)} \int_{0}^{\infty} e^{-u} \, u^{- t - 1}  F_{{1}:{0}; {0}} ^{{1}:{1}; {1}}\left(\begin{array}{ccc}
a : & b_1 ; & b_2\\
c: & - ; & - 
\end{array}; (-u)^k \,  x,  (-u)^k \,  y\right) du.
\end{align}
\subsection{Differential formulae}
The following  differential formulae are satisfied by discrete Appell function $\mathcal{F}^{(2)}_1$:
\begin{align}
	& (\theta)^r \mathcal{F}^{(2)}_1(a, b_1, b_2; c; t, k, x, y) \nonumber\\
	& = \frac{(-1)^{rk} \, (a)_r \, (b_1)_r \, (-t)_{rk} \, x^r}{(c)_r} \, \mathcal{F}^{(2)}_1(a + r, b_1 + r, b_2; c + r; t - rk, k, x, y);\\
	& (\phi)^r \mathcal{F}^{(2)}_1(a, b_1, b_2; c; t, k, x, y) \nonumber\\
	& = \frac{(-1)^{rk} \, (a)_r \, (b_2)_r \, (-t)_{rk} \, y^r}{(c)_r} \, \mathcal{F}^{(2)}_1(a + r, b_1, b_2 + r; c + r; t - rk, k, x, y);\\
	& \left(\frac{\partial}{\partial x}\right)^r \left[x^{b_1 + r - 1} \mathcal{F}^{(2)}_1(a, b_1, b_2; c; t, k, x, y)\right]\nonumber\\
	& = x^{b_1 - 1} \, (b_1)_r \, \mathcal{F}^{(2)}_1(a, b_1 + r, b_2; c; t, k, x, y);\\
	& \left(\frac{\partial}{\partial y}\right)^r [y^{b_2 + r - 1} \mathcal{F}^{(2)}_1(a, b_1, b_2; c; t, k, x, y)]\nonumber\\
	& = y^{b_2 - 1} \, (b_2)_r \, \mathcal{F}^{(2)}_1(a, b_1, b_2 + r; c; t, k, x, y);\\
	& \left(\frac{\partial}{\partial x}\right)^r [x^{a + r - 1} \mathcal{F}^{(2)}_1(a, b_1, b_2; c; t, k, x, xy)]\nonumber\\
	& = x^{a - 1} \, (a)_r \, \mathcal{F}^{(2)}_1(a + r, b_1, b_2; c; t, k, x, xy);\\
	& \left(\frac{\partial}{\partial y}\right)^r [y^{a + r - 1} \mathcal{F}^{(2)}_1(a, b_1, b_2; c; t, k, xy, y)]\nonumber\\
	& = y^{a - 1} \, (a)_r \, \mathcal{F}^{(2)}_1(a + r, b_1, b_2; c; t, k, xy, y);\\
	& \left(\frac{\partial}{\partial x}\right)^r [x^{c - 1} \mathcal{F}^{(2)}_1(a, b_1, b_2; c; t, k, x, xy)]\nonumber\\
	& = (-1)^r \, (1 - c)_r \, x^{c - r - 1} \mathcal{F}^{(2)}_1(a, b_1, b_2; c - r; t, k, x, xy);\\
	& \left(\frac{\partial}{\partial y}\right)^r [y^{c - 1} \mathcal{F}^{(2)}_1(a, b_1, b_2; c; t, k, xy, y)]\nonumber\\
	& = (-1)^r \, y^{c - r - 1} \, (1 - c)_r \, \mathcal{F}^{(2)}_1(a, b_1, b_2; c - r; t, k, xy, y).
\end{align}
\subsection{Finite and infinite summation formulae}
Some finite and infinite summation formulas for the discrete Appell function $\mathcal{F}^{(2)}_1$ can be obtained as
\begin{theorem} Following summation formulas for the discrete Appell function $\mathcal{F}^{(2)}_1$ hold:
	\begin{align}
		& \mathcal{F}^{(2)}_1(a, b_1 + r, b_2; c; t, k, x, y)\nonumber\\
		& = \sum_{s = 0}^{r} {r \choose s} \frac{(a)_s \, (-1)^{sk} \, (-t)_{sk}}{(c)_s} \, x^s \, \mathcal{F}^{(2)}_1(a + s, b_1 + s, b_2; c + s; t - sk,  k, x, y);\\
		& \mathcal{F}^{(2)}_1(a, b_1, b_2 + r; c; t, k, x, y)\nonumber\\
		& = \sum_{s = 0}^{r} {r \choose s} \frac{(a)_s \, (-1)^{sk} \, (-t)_{sk}}{(c)_s} \, y^s \, \mathcal{F}^{(2)}_1(a + s, b_1, b_2 + s; c + s; t - sk, k, x, y); \\
			&\sum_{r = 0}^{\infty} \frac{(a)_r}{r !} \, z^r \, \mathcal{F}^{(2)}_1 (a + r, b_1, b_2; c ; t, k, x, y) = (1 - z)^{-a} \, \mathcal{F}^{(2)}_1 \left(a, b_1, b_2; c; t, k, \frac{x}{1 - z}, \frac{y}{1 - z}\right); \\
			&\sum_{r = 0}^{\infty} \frac{(b_1)_r}{r !} \, z^r \, \mathcal{F}^{(2)}_1 (a, b_1 + r, b_2; c ; t, k, x, y) = (1 - z)^{-b_1} \, \mathcal{F}^{(2)}_1 \left(a, b_1, b_2; c; t, k, \frac{x}{1 - z}, y\right); \\
			&\sum_{r = 0}^{\infty} \frac{(b_2)_r}{r !} \, z^r \, \mathcal{F}^{(2)}_1 (a, b_1, b_2 + r; c ; t, k, x, y) = (1 - z)^{-b_2} \, \mathcal{F}^{(2)}_1 \left(a, b_1, b_2; c; t, k, x, \frac{y}{1 - z}\right). 
	\end{align}
\end{theorem} 
The following recursion formulas hold for the discrete Appell function $ \mathcal{F}^{(2)}_1$. 
\begin{align}
	& \mathcal{F}^{(2)}_1 (a + s, b_1, b_2; c ; t, k, x, y) \nonumber\\
	& = \mathcal{F}^{(2)}_1 (a, b_1, b_2; c ; t, k, x, y)\nonumber\\
	& \quad + \frac{(-1)^k \, (-t)_k \, b_1 \, x}{c} \sum_{r = 1}^{s} \mathcal{F}^{(2)}_1 (a + r, b_1 + 1, b_2; c + 1; t - k, k, x, y) \nonumber\\
	& \quad + \frac{(-1)^k \, (-t)_k \, b_2 \, y}{c}  \sum_{r = 1}^{s} \mathcal{F}^{(2)}_1 (a + r, b_1, b_2 + 1; c + 1; t - k, k, x, y);\\
	& \mathcal{F}^{(2)}_1 (a - s, b_1, b_2; c ; t, k, x, y) \nonumber\\
	& = \mathcal{F}^{(2)}_1 (a, b_1, b_2; c ; t, k, x, y) \nonumber\\
	& \quad - \frac{(-1)^k \, (-t)_k \, b_1 \, x}{c}  \sum_{r = 0}^{s - 1} \mathcal{F}^{(2)}_1 (a - r, b_1 + 1, b_2; c + 1; t - k, k, x, y) \nonumber\\
	& \quad - \frac{(-1)^k \, (-t)_k \, b_2 \, y}{c}  \sum_{r = 0}^{s - 1} \mathcal{F}^{(2)}_1 (a - r, b_1, b_2 + 1; c + 1; t - k, k, x, y);\\
	& \mathcal{F}^{(2)}_1 (a, b_1 + s, b_2; c ; t, k, x, y) \nonumber\\
	& = \mathcal{F}^{(2)}_1 (a, b_1, b_2; c ; t, k, x, y) \nonumber\\
	& \quad + \frac{(-1)^k \, (-t)_k \, a \, x}{c} \sum_{r = 1}^{s} \mathcal{F}^{(2)}_1 (a + 1, b_1 + r, b_2; c + 1; t - k,  k, x, y);\\
	& \mathcal{F}^{(2)}_1 (a, b_1 - s, b_2; c ; t, k, x, y) \nonumber\\
	& = \mathcal{F}^{(2)}_1 (a, b_1, b_2; c ; t, k, x, y) \nonumber\\
	& \quad - \frac{(-1)^k \, (-t)_k \, a \, x}{c}  \sum_{r = 0}^{s - 1} \mathcal{F}^{(2)}_1 (a + 1, b_1 - r, b_2; c + 1; t - k, k, x, y);\\
	& \mathcal{F}^{(2)}_1 (a, b_1, b_2; c - s; t, k, x, y) \nonumber\\
	& = \mathcal{F}^{(2)}_1 (a, b_1, b_2; c ; t, k, x, y) \nonumber\\
	& \quad + (-1)^k \, (-t)_k \, a \, b_1 \, x \sum_{r = 1}^{s} \frac{\mathcal{F}^{(2)}_1 (a + 1, b_1 + 1, b_2; c + 2 - r; t - k, k, x, y)}{(c - r) \, (c - r + 1)} \nonumber\\
	& \quad + (-1)^k \, (-t)_k \, a \, b_2 \, y  \sum_{r = 1}^{s} \frac{\mathcal{F}^{(2)}_1 (a + 1, b_1, b_2 + 1; c + 2 - r; t - k, k, x, y)}{(c - r) \, (c - r + 1)}.
\end{align}
Finally, we list differential recursion formulae satisfied by $\mathcal{F}^{(2)}_1$. Using the list of simple differential relations as follows: 
\begin{align}
	&	a \, \mathcal{F}^{(2)}_1 (a + 1)  = (a + \theta + \phi) \, \mathcal{F}^{(2)}_1;\\
	& (a + \theta + \phi - 1) \, \mathcal{F}^{(2)}_1 (a - 1)  =	(a - 1) \, \mathcal{F}^{(2)}_1;\\
	&	b_1 \, \mathcal{F}^{(2)}_1 (b_1 + 1)  = (b_1 + \theta) \, \mathcal{F}^{(2)}_1;\\
	&(b_1 + \theta - 1) \, \mathcal{F}^{(2)}_1 (b_1 - 1)  =	(b_1 - 1) \, \mathcal{F}^{(2)}_1;\\
	&	b_2 \, \mathcal{F}^{(2)}_1 (b_2 + 1)  = (b_2 + \phi) \, \mathcal{F}^{(2)}_1;\\
	& (b_2  + \phi - 1) \, \mathcal{F}^{(2)}_1 (b_2 - 1)  =	(b_2 - 1) \, \mathcal{F}^{(2)}_1;\\
	&	(c - 1) \, \mathcal{F}^{(2)}_1 (c - 1)  = (c + \theta + \phi - 1) \, \mathcal{F}^{(2)}_1;\\
	& (c + \theta + \phi) \, \mathcal{F}^{(2)}_1 (c + 1)  =	c \, \mathcal{F}^{(2)}_1.
\end{align}
%On combining any two of the above relations, 
we get first order or second order differential recursion relations as
\begin{align}
	& a \, (a - 1) \, \mathcal{F}^{(2)}_1 (a + 1) - (a + \theta + \phi) \, (a + \theta + \phi - 1) \mathcal{F}^{(2)}_1 (a - 1) = 0;\\
	& a \, (b_1 - 1) \, \mathcal{F}^{(2)}_1 (a + 1) - (a + \theta + \phi) \, (b_1 + \theta  - 1) \mathcal{F}^{(2)}_1 (b_1 - 1) = 0;\\
	& a \, (b_2 - 1) \, \mathcal{F}^{(2)}_1 (a + 1) - (a + \theta + \phi) \, (b_2 + \phi - 1) \mathcal{F}^{(2)}_1 (b_2 - 1) = 0;\\
	& a \, c \, \mathcal{F}^{(2)}_1 (a + 1) - (a + \theta + \phi) \, (c + \theta + \phi) \mathcal{F}^{(2)}_1 (c + 1) = 0;\\
	& a \, (b_1 + \theta) \, \mathcal{F}^{(2)}_1 (a + 1) - b_1 \, (a + \theta + \phi) \, \mathcal{F}^{(2)}_1 (b_1 + 1) = 0;\\
	& a \, (b_2 + \phi) \, \mathcal{F}^{(2)}_1 (a + 1) - b_2 \, (a + \theta + \phi) \, \mathcal{F}^{(2)}_1 (b_2 + 1) = 0;\\
	& a \, (c + \theta + \phi - 1) \, \mathcal{F}^{(2)}_1 (a + 1) - (c - 1) \, (a + \theta + \phi) \, \mathcal{F}^{(2)}_1 (c - 1) = 0;\\
	&(a + \theta + \phi - 1) \, (b_1 + \theta) \, \mathcal{F}^{(2)}_1 (a - 1) - b_1 \, (a - 1)  \, \mathcal{F}^{(2)}_1 (b_1 + 1) = 0;\\
	&	(a + \theta + \phi - 1) \, (b_2 + \phi) \, \mathcal{F}^{(2)}_1 (a - 1) - b_2 \, (a - 1)  \, \mathcal{F}^{(2)}_1 (b_2 + 1) = 0;\\
	&	(a + \theta + \phi - 1) \, (c + \theta + \phi - 1) \, \mathcal{F}^{(2)}_1 (a - 1) - (c - 1) \, (a - 1)  \, \mathcal{F}^{(2)}_1 (c - 1) = 0;\\
	& (b_1 - 1) \,	(a + \theta + \phi - 1) \,  \mathcal{F}^{(2)}_1 (a - 1) -  (a - 1)  \, (b_1 + \theta - 1) \mathcal{F}^{(2)}_1 (b_1 - 1) = 0;\\
	& (b_2 - 1) \,	(a + \theta + \phi - 1) \,  \mathcal{F}^{(2)}_1 (a - 1) -  (a - 1)  \, (b_2 + \phi - 1) \mathcal{F}^{(2)}_1 (b_2 - 1) = 0;\\
	& c \,	(a + \theta + \phi - 1) \,  \mathcal{F}^{(2)}_1 (a - 1) -  (a - 1)  \, (c + \theta + \phi) \mathcal{F}^{(2)}_1 (c + 1) = 0;\\
	& b_1 \,	(b_1 - 1) \,  \mathcal{F}^{(2)}_1 (b_1 + 1) -  (b_1 + \theta)  \, (b_1 + \theta - 1) \mathcal{F}^{(2)}_1 (b_1 - 1) = 0;\\
	& b_1 \,	(b_2 + \phi) \,  \mathcal{F}^{(2)}_1 (b_1 + 1) -  b_2  \, (b_1 + \theta ) \mathcal{F}^{(2)}_1 (b_2 + 1) = 0;\\
	& b_1 \,	(b_2 - 1) \,  \mathcal{F}^{(2)}_1 (b_1 + 1) -  (b_1 + \theta)  \, (b_2 + \phi - 1) \mathcal{F}^{(2)}_1 (b_2 - 1) = 0;\\
	& b_1 \,	(c + \theta + \phi - 1) \,  \mathcal{F}^{(2)}_1 (b_1 + 1) -  (c - 1)  \, (b_1 + \theta) \mathcal{F}^{(2)}_1 (c - 1) = 0;\\
	& b_1 \,	c \,  \mathcal{F}^{(2)}_1 (b_1 + 1) -  (c + \theta +\phi)  \, (b_1 + \theta) \mathcal{F}^{(2)}_1 (c + 1) = 0;\\
	& b_2 \,	(b_1 - 1) \,  \mathcal{F}^{(2)}_1 (b_2 + 1) -  (b_2 + \phi)  \, (b_1 + \theta - 1) \mathcal{F}^{(2)}_1 (b_1 - 1) = 0;\\
	& b_2 \,	(b_2 - 1) \,  \mathcal{F}^{(2)}_1 (b_2 + 1) -  (b_2 + \phi)  \, (b_2 + \phi - 1) \mathcal{F}^{(2)}_1 (b_2 - 1) = 0;\\
	& b_2 \,	(c + \theta + \phi - 1) \,  \mathcal{F}^{(2)}_1 (b_2 + 1) -  (c - 1)  \, (b_2 + \phi) \mathcal{F}^{(2)}_1 (c - 1) = 0;\\
	& b_2 \,	c \,  \mathcal{F}^{(2)}_1 (b_2 + 1) -  (c + \theta +\phi)  \, (b_2 + \phi) \mathcal{F}^{(2)}_1 (c + 1) = 0;\\
	& (b_2 - 1) \,	(b_1 + \theta - 1) \,  \mathcal{F}^{(2)}_1 (b_1 - 1) -  (b_1 - 1)  \, (b_2 + \phi - 1) \mathcal{F}^{(2)}_1 (b_2 - 1) = 0;\\
	& (b_1 + \theta - 1) \,	(c + \theta + \phi - 1) \,  \mathcal{F}^{(2)}_1 (b_1 - 1) -  (c - 1)  \, (b_1 - 1) \mathcal{F}^{(2)}_1 (c - 1) = 0;\\
	& 	c \, (b_1 + \theta - 1) \, \mathcal{F}^{(2)}_1 (b_1 - 1) - (b_1 - 1)  (c + \theta +\phi)  \,  \mathcal{F}^{(2)}_1 (c + 1) = 0;\\
	& (b_2 + \phi - 1) \,	(c + \theta + \phi - 1) \,  \mathcal{F}^{(2)}_1 (b_2 - 1) -  (c - 1)  \, (b_2 - 1) \mathcal{F}^{(2)}_1 (c - 1) = 0;\\
	& 	c \, (b_2 + \phi - 1) \, \mathcal{F}^{(2)}_1 (b_2 - 1) - (b_2 - 1)  (c + \theta +\phi)  \,  \mathcal{F}^{(2)}_1 (c + 1) = 0;\\
	& 	c \, (c - 1) \, \mathcal{F}^{(2)}_1 (c - 1) - (c + \theta + \phi - 1)  (c + \theta +\phi)  \,  \mathcal{F}^{(2)}_1 (c + 1) = 0.
\end{align}
Similarly using  the difference-differential relations: 
\begin{align}
	&	a \, \mathcal{F}^{(2)}_1 (a + 1)  = \left(a + \frac{1}{k} \Theta_{t} \right) \, \mathcal{F}^{(2)}_1;\\
	& \left(a + \frac{1}{k} \Theta_{t} - 1\right) \, \mathcal{F}^{(2)}_1 (a - 1)  =	(a - 1) \, \mathcal{F}^{(2)}_1;\\
	&	b_1 \, \mathcal{F}^{(2)}_1 (b_1 + 1)  = \left(b_1 + \theta \right) \, \mathcal{F}^{(2)}_1;\\
	&\left(b_1 + \theta - 1\right) \, \mathcal{F}^{(2)}_1 (b_1 - 1)  =	(b_1 - 1) \, \mathcal{F}^{(2)}_1;\\
	&	b_2 \, \mathcal{F}^{(2)}_1 (b_2 + 1)  = \left(b_2 + \phi \right) \, \mathcal{F}^{(2)}_1;\\
	& \left(b_2 + \phi - 1\right) \, \mathcal{F}^{(2)}_1 (b_2 - 1)  =	(b_2 - 1) \, \mathcal{F}^{(2)}_1;\\
	&	(c - 1) \, \mathcal{F}^{(2)}_1 (c - 1)  = \left(c + \frac{1}{k} \Theta_{t} - 1\right) \, \mathcal{F}^{(2)}_1;\\
	& \left(c + \frac{1}{k} \Theta_{t}\right) \, \mathcal{F}^{(2)}_1 (c + 1)  =	c \, \mathcal{F}^{(2)}_1,
\end{align}
we get the difference-differential recursion relations as: 
\begin{align}
	& a \, (a - 1) \, \mathcal{F}^{(2)}_1 (a + 1) - \left(a + \frac{1}{k} \Theta_{t}\right) \, \left(a + \frac{1}{k} \Theta_{t} - 1\right) \mathcal{F}^{(2)}_1 (a - 1) = 0;\\
	& a \, (b_1 - 1) \, \mathcal{F}^{(2)}_1 (a + 1)  - \left(a + \frac{1}{k} \Theta_{t} \right) \, \left(b_1 + \theta  - 1\right) \mathcal{F}^{(2)}_1 (b_1 - 1) = 0;\\
	& a \, (b_2 - 1) \, \mathcal{F}^{(2)}_1 (a + 1) - \left(a + \frac{1}{k} \Theta_{t}\right) \, \left(b_2 + \phi - 1\right) \mathcal{F}^{(2)}_1 (b_2 - 1) = 0;\\
	& a \, c \, \mathcal{F}^{(2)}_1 (a + 1) - \left(a + \frac{1}{k} \Theta_{t}\right) \, \left(c + \frac{1}{k} \Theta_{t} \right) \mathcal{F}^{(2)}_1 (c + 1) = 0;\\
	& a \, \left(b_1 + \theta\right) \, \mathcal{F}^{(2)}_1 (a + 1) - b_1 \, \left(a + \frac{1}{k} \Theta_{t}\right) \, \mathcal{F}^{(2)}_1 (b_1 + 1) = 0;\\
	& a \, \left(b_2 + \phi\right) \, \mathcal{F}^{(2)}_1 (a + 1) - b_2 \, \left(a + \frac{1}{k} \Theta_{t}\right) \, \mathcal{F}^{(2)}_1 (b_2 + 1) = 0;\\
	& a \, \left(c + \frac{1}{k} \Theta_{t} - 1\right) \, \mathcal{F}^{(2)}_1 (a + 1) - (c - 1) \, \left(a + \frac{1}{k} \Theta_{t}\right) \, \mathcal{F}^{(2)}_1 (c - 1) = 0;\\
	&\left(a + \frac{1}{k} \Theta_{t} - 1\right) \, \left(b_1 + \theta \right) \, \mathcal{F}^{(2)}_1 (a - 1) - b_1 \, (a - 1)  \, \mathcal{F}^{(2)}_1 (b_1 + 1) = 0;\\
	&	\left(a + \frac{1}{k} \Theta_{t} - 1\right) \, \left(b_2 + \phi\right) \, \mathcal{F}^{(2)}_1 (a - 1) - b_2 \, (a - 1)  \, \mathcal{F}^{(2)}_1 (b_2 + 1) = 0;\\
	&	\left(a + \frac{1}{k} \Theta_{t} - 1\right) \, \left(c + \frac{1}{k} \Theta_{t} - 1\right) \, \mathcal{F}^{(2)}_1 (a - 1) - (c - 1) \, (a - 1)  \, \mathcal{F}^{(2)}_1 (c - 1) = 0;\\
	& (b_1 - 1) \,	\left(a + \frac{1}{k} \Theta_{t} - 1\right) \,  \mathcal{F}^{(2)}_1 (a - 1)  -  (a - 1)  \, \left(b_1 + \theta  - 1\right) \mathcal{F}^{(2)}_1 (b_1 - 1) = 0;\\
	& (b_2 - 1) \,	\left(a + \frac{1}{k} \Theta_{t} - 1\right) \,  \mathcal{F}^{(2)}_1 (a - 1)  -  (a - 1)  \, \left(b_2 + \phi - 1\right) \mathcal{F}^{(2)}_1 (b_2 - 1) = 0;\\
	& c \,	\left(a + \frac{1}{k} \Theta_{t} - 1\right) \,  \mathcal{F}^{(2)}_1 (a - 1) -  (a - 1)  \, \left(c + \frac{1}{k} \Theta_{t}\right) \mathcal{F}^{(2)}_1 (c + 1) = 0;\\
	& b_1 \,	\left(c + \frac{1}{k} \Theta_{t} -  1\right) \,  \mathcal{F}^{(2)}_1 (b_1 + 1) -  (c - 1)  \, \left(b_1 + \theta \right) \mathcal{F}^{(2)}_1 (c - 1) = 0;\\
	& b_1 \,	c \,  \mathcal{F}^{(2)}_1 (b_1 + 1) -  \left(c + \frac{1}{k} \Theta_{t}\right)  \, \left(b_1 + \theta\right) \mathcal{F}^{(2)}_1 (c + 1) = 0;\\
	& b_2 \,	\left(c + \frac{1}{k} \Theta_{t} - 1\right) \,  \mathcal{F}^{(2)}_1 (b_2 + 1) -  (c - 1)  \, \left(b_2 + \phi\right) \mathcal{F}^{(2)}_1 (c - 1) = 0;\\
	& b_2 \,	c \,  \mathcal{F}^{(2)}_1 (b_2 + 1) -  \left(c + \frac{1}{k} \Theta_{t}\right)  \, \left(b_2 + \phi\right) \mathcal{F}^{(2)}_1 (c + 1) = 0;\\
	& \left(b_1 + \theta - 1\right) \,	\left(c + \frac{1}{k} \Theta_{t} - 1\right) \,  \mathcal{F}^{(2)}_1 (b_1 - 1) -  (c - 1)  \, (b_1 - 1) \mathcal{F}^{(2)}_1 (c - 1) = 0;\\
	& 	c \, \left(b_1 + \theta - 1\right) \, \mathcal{F}^{(2)}_1 (b_1 - 1) - (b_1 - 1)  \left(c + \frac{1}{k} \Theta_{t}\right)  \,  \mathcal{F}^{(2)}_1 (c + 1) = 0;\\
	& \left(b_2 + \phi - 1\right) \,	\left(c + \frac{1}{k} \Theta_{t} - 1\right) \,  \mathcal{F}^{(2)}_1 (b_2 - 1) -  (c - 1)  \, (b_2 - 1) \mathcal{F}^{(2)}_1 (c - 1) = 0;\\
	& 	c \, \left(b_2 + \phi - 1\right) \, \mathcal{F}^{(2)}_1 (b_2 - 1) - (b_2 - 1)  \left(c + \frac{1}{k} \Theta_{t}\right)  \,  \mathcal{F}^{(2)}_1 (c + 1) = 0;\\
	& 	c \, (c - 1) \, \mathcal{F}^{(2)}_1 (c - 1) - \left(c + \frac{1}{k} \Theta_{t} - 1\right)  \left(c + \frac{1}{k} \Theta_{t}\right)  \,  \mathcal{F}^{(2)}_1 (c + 1) = 0.
\end{align}

\section{Conclusion}  
%We all know that Appell gives a list of four two variable functions $F_1$, $F_2$, $F_3$ and $F_4$. It was not possible to list all the findings for discrete analogues of these four functions in a single paper and hence we decided to collect the findings of each Appell function separately. In this article we recorded the analysis of discrete analogues of Appell function $F_1$ and for rest three, we have written articles separately. 
We have discussed two distinct discrete forms of Appell function $F_1$ \emph{viz.} $\mathcal{F}^{(1)}_1$ and $\mathcal{F}^{(2)}_1$ in a detailed manner. Two particular cases are of special interest. These are the first discrete form $\mathcal{F}^{(1)}_1$ with 
%, it can be immediately concluded that for
$k_1 = k_2 = k$:
\begin{align}
 &\mathcal{F}^{(1)}_1(a, b_1, b_2; c; t_1, t_2, k, x, y)\nonumber\\ 
	& = \sum_{m,n\geq0} \frac{(a)_{m+n} \, (b_1)_m \, (b_2)_n \, (-1)^{(m + n) k} \,  (-t_1)_{mk} \, (-t_2)_{nk}}{ (c)_{m+n} \, m! \, n!} \ x^m \, y^n
\end{align} 
and with $t_1 = t_2 = t$:
\begin{align}
&\mathcal{F}^{(1)}_1(a, b_1, b_2; c; t, k_1, k_2, x, y)\nonumber\\ 
	& = \sum_{m,n \geq 0} \frac{(a)_{m+n} \, (b_1)_m \, (b_2)_n \, (-1)^{m k_1} \, (-1)^{n k_2} \, (-t)_{mk_1} \, (-t)_{nk_2}}{ (c)_{m+n} \, m! \, n!} \ x^m \, y^n.
\end{align}
These two functions can also be considered as the discrete analogues of Appell function $F_1$. Most of the results related to these discrete functions can be deduced from the discrete function $\mathcal{F}^{(1)}_1$. However, there is a scope for studying these functions separately. Furthermore, we have presented the discrete analogues of Humbert function $\phi_1$, $\phi_2$ and $\phi_3$ as the limiting cases of $\mathcal{F}^{(1)}_1$ and $\mathcal{F}^{(2)}_1$. So, the analytical properties of these discrete Humbert functions can be deduced by taking the limit in the respective properties of discrete Appell functions. Besides, study of the results and properties of these discrete Humbert function in comparative way will be an interesting problem. 

As discussed in the paper, the particular cases of the functions $\mathcal{F}^{(1)}_1$ and $\mathcal{F}^{(2)}_1$ reduce to the Appell function $F_1$ and the Kamp\'e de F\'eriet functions. Therefore the particular cases of results obtained for $\mathcal{F}^{(1)}_1$ and $\mathcal{F}^{(2)}_1$ will lead to the results for $F_1$ and  Kamp\'e de F\'eriet functions.  Most of the results obtained for $F_1$ in this way are already available in the literature but the differential recursion formulas are the new ones. Similarly, the results that can be determined for Kamp\'e de F\'eriet functions are also believed to be new. 

Appell has given four two variable functions $F_1$, $F_2$, $F_3$ and $F_4$. It was not possible to list all the findings for discrete analogues of these four functions in a single paper and hence we decided to collect the findings of each discrete Appell function separately. In this article, we have recorded the analysis of discrete analogues of Appell function $F_1$ and for discrete analogues of remaining three Appell functions, we have written articles separately.


\begin{thebibliography}{99}	
%\bibitem{al} A.\,Altin, B.\,Cekim, R.\,Sahin, \textit{On the matrix versions of Appell hypergeometric functions}. Quaest. Math. 37 (2014), no. 1, 31--38.

\bibitem{bc1} M.\,Bohner, T.\,Cuchta, \textit{The Bessel difference equation.} Proc. Amer. Math. Soc. 145 (2017), no. 4, 1567--1580. 

\bibitem{bc4} M.\,Bohner, T.\,Cuchta, \textit{The generalized hypergeometric difference equation.} Demonstr. Math., 51(2018), no. 1, 62--75.

\bibitem{bc2} T.\,Cuchta, D.\,Grow, N.\,Wintz, \textit{Divergence criteria for matrix generalized hypergeometric series.} Proc. Amer. Math. Soc. 150 (2022), no. 3, 1235--1240. 

\bibitem{bc3} T.\,Cuchta, R.\,Luketic, \textit{Discrete hypergeometric Legendre polynomials.} Mathematics 9 (2021), 2546.

\bibitem{bc5} T.\,Cuchta, D.\,Grow, N.\,Wintz, \textit{Discrete matrix hypergeometric functions.} J. Math. Anal. Appl. 518 (2023), no. 2, Paper No. 126716, 14~pp.

\bibitem{kdf} P.\,Appell, J.\,Kamp\'e de F\'eriet, \textit{Fonctions hyperg\'eometriques et hypersph\'eriques. Polyn\^omes d'Hermite}, Paris, Gauthier-Villars, 1926.

\bibitem{ds1} R.\,Dwivedi, V.\,Sahai, \textit{On the hypergeometric matrix functions of two variables.} Linear Multilinear Algebra 66 (2018), no. 9, 1819--1837.

\bibitem{ds2} R.\,Dwivedi, V.\,Sahai, \textit{On the basic hypergeometric matrix functions of two variables}.  Linear Multilinear Algebra 67 (2019), no. 1, 1--19. 

\bibitem{ds5} R.\,Dwivedi, V.\,Sahai, \textit{A note on the Appell matrix functions}, Quaest. Math.  43 (2020), no. 3, 321-334.   %\url{https://doi.org/10.2989/16073606.2019.1577309}.

\bibitem{emo} A.\,Erd\'elyi, W.\,Magnus, F.\,Oberhettinger, F.\,G.\,Tricomi, \textit{Higher Transcendental Functions}, Vol. I, McGraw-Hill, New York, London, 1953.

\bibitem{gr} G.\,Gasper, M.\,Rahman, \textit{Basic hypergeometric series},  Second edition. Encyclopedia of Mathematics and its Applications, 96. Cambridge University Press, Cambridge, 2004.

\bibitem{ph} P.\,Humbert, \textit{The confluent hypergeometric functions of two variables.} Proc. Roy. Soc. Edinburgh, 41 (1920-21), 73--96.	
%\bibitem{bc40} J.\,L.\,Burchnall, T.\,W.\,Chaundy, \emph{Expansions of Appell's double hypergeometric functions}, Quart. J. Math., Oxford Ser. 11, (1940). 249--270. 
%		
%\bibitem{cm} A.\,G.\,Constantine, R.\,J.\,Muirhead, \textit{Partial differential equations for hypergeometric functions of two argument matrices}, J. Multivariate Anal. 2 (1972), 332--338.			
%\bibitem{ds57} N.\,Dunford, J.\,Schwartz, \textit{Linear Operators}, Part-I, New York: Addison-Wesley, 1957.	
%	
%\bibitem{gf} G.B.\,Folland, \textit{Fourier Analysis and its Applications}, Wadsworth and Brooks, Pacific Grove, CA 1992.		
%		
%\bibitem{gl} G.\,Golub, C.F.\,van Loan, \textit{Matrix Computation}, The Johns Hopkins Univ. Press, Baltimore, MA, 1989.
%	
%\bibitem{hc} I.\,Higueras, B.\,Garcia-Celayeta, \textit{Logarithmic norms for matrix pencils}. SIAM J. Matrix Anal. Appl. 20(1999), 646-666.
%	
%\bibitem{hl04} G.D.\,Hu, M.\,Liu, \textit{The weighted logarithmic matrix norm and bounds of the matrix exponential.}, Linear Algebra Appl. 2004;390:145--154.
%		
%\bibitem{atj} A.T.\,James, Special functions of matrix and single argument in statistics. Theory and application of special functions (Proc. Advanced Sem., Math. Res. Center, Univ. Wisconsin, Madison, Wis., 1975), pp. 497--520. Math. Res. Center, Univ. Wisconsin, Publ. No. 35, Academic Press, New York, 1975.
%	
%\bibitem{jc} L.\,J\'odar, R.\,Company, \textit{Hermite matrix polynomials and second order matrix differential equations}, Approx. Theory Appl. (N.S.) 12 (1996), no. 2, 20--30.
%	
%\bibitem{jcn} L.~J\'odar, R.~Company, E.~Navarro, \textit{Laguerre matrix polynomials and systems of second-order differential equations}, Appl. Numer. Math. 15 (1994), no. 1, 53--63.
%		
%\bibitem{jcp} L.\,J\'odar, R.\,Company, E.\,Ponsoda \textit{Orthogonal matrix polynomials and systems of second order differential equations}, Differential Equations Dynam. Systems 3 (1995), no. 3, 269--288.
%
%\bibitem{jjc98a} L.\,J\'odar, J.C.\,Cort\'es, \textit{Some properties of gamma and beta matrix functions}, Appl. Math. Lett. 11 (1998), no. 1, 89--93.
%
%\bibitem{jjc98b}  L.\,J\'odar, J.C.\,Cort\'es, \textit{On the hypergeometric matrix function}, Proceedings of the VIIIth Symposium on Orthogonal Polynomials and Their Applications (Seville, 1997). J. Comput. Appl. Math. 99 (1998), no. 1-2, 205--217. 

\bibitem{am} A.\,M.\,Mathai, \textit{ Appell's and Humbert's functions of matrix arguments.} Linear Algebra Appl. 183 (1993), 201--221.

\bibitem{am1} A.\,M.\,Mathai, \textit{ Jacobians of matrix transformations and functions of matrix argument}, World Scientific Publishing Co., Inc., River Edge, NJ, 1997.

\bibitem{mj} James\,A.\,Mullen, \textit{The differential recursion formulae for Appell's hypergeometric functions of two variables}. SIAM J. Appl. Math. 14 (1966), 1152--1163.
		
\bibitem{nn} D.\,K.\,Nagar, S.\,Nadarajah, \textit{Appell's hypergeometric functions of matrix arguments}, Integral Transforms Spec. Funct. 28 (2017), no. 2, 91--112.
			
\bibitem{edr} E.\,D.\,Rainville, \emph{Special Functions}, Chelsea, New York, 1960.
			
%\bibitem{sd06} J.\,Sastre, E.\,Defez, \textit{On the asymptotics of Laguerre matrix polynomials for large $x$ and $n$}, Appl. Math. Lett. 19 (2006), no. 8, 721--727.	
%		 
%\bibitem{sj03} J.\,Sastre, L.\,J\'odar, \textit{Asymptotics of the modified Bessel and the incomplete gamma matrix functions}, Appl. Math. Lett. 16 (2003), no. 6, 815--820.	  
%
%\bibitem{js} J.\,B.\,Seaborn, \textit{Hypergeometric Functions and their Applications}, Springer, New York, 1991.	
%	 
%\bibitem{te05} L.N.\,Trefethen, M.\,Embree, \textit{Spectra and pseudospectra: the behaviour of nonnormal matrices and operators.} Princeton, NJ: Princeton University Press; 2005. 	
%	
%\bibitem{vl} C.\,Van Loan, \textit{The sensitivity of the matrix exponential}, SIAM J. Numer. Anal. 14 (1977), no. 6, 971--981.	

\bibitem{sk} H.\,M.\,Srivastava, P.\,W.\,Karlsson, \textit{Multiple Gaussian Hypergeometric Series}, Ellis Horwood Limited, Chichester, 1985.

\bibitem{bts} M.\,Tripathi, N.\,Saikia,  R.\,Barman, \textit{Appell's hypergeometric series over finite fields.} Int. J. Number Theory 16  (2020), no. 4, 673--692.
\end{thebibliography}
\end{document}